\documentclass[11pt, reqno, twoside, letterpaper]{amsart}

\usepackage{z-Paperstyle}
\usepackage{tikz}
\usepackage{tikz-cd}

\draftfalse

\makeatletter \def\section{\@startsection{section}{1}%
\z@{1.2\linespacing\@plus\linespacing}{.5\linespacing}%
{\normalfont\Large\scshape\centering}} \makeatother

\newcommand{\vn}{\mathbf{n}}
\newcommand{\va}{\boldsymbol{\alpha}}
\newcommand{\vz}{\mathbf{z}}
\newcommand{\ve}{\mathbf{e}}
\newcommand{\vu}{\mathbf{u}}
\newcommand{\vv}{\mathbf{v}}
\newcommand{\vzero}{\mathbf{0}}
\DeclareMathOperator{\rk}{rank}
\DeclareMathOperator{\Gal}{Gal}
\DeclareMathOperator{\Frob}{Frob}
\DeclareMathOperator{\disc}{disc}
\newcommand{\house}[1]{\lVert #1\rVert}
\newcommand{\Nm}{\mathrm{N}}
\newcommand{\fp}{\mathfrak{p}}
\newcommand{\fa}{\mathfrak{a}}
\newcommand{\fb}{\mathfrak{b}}
\newcommand{\fc}{\mathfrak{c}}
\newcommand{\fd}{\mathfrak{d}}

\newcommand{\dk}{{\mathtt{n}_k}}
\newcommand{\dK}{{\mathtt{n}_K}}
\newcommand{\auxp}{\ell}
\newcommand{\rou}{\xi}
\newcommand{\wgt}{H}
\newcommand{\bxs}{N}
\newcommand{\dgb}{D}
\newcommand{\expC}{C}
\newcommand{\XX}{X}
\newcommand{\hB}{B}
\newcommand{\lam}{\lambda}
\newcommand{\lamK}{\lambda_{\Kf}}
\newcommand{\hgt}{T}
\newcommand{\galKk}{\mathsf{G}}
\newcommand{\galL}{\widetilde{\mathsf{G}}}
\newcommand{\Kf}{K}
\newcommand{\lat}{\Gamma}
\newcommand{\poly}{\wp}
\newcommand{\idx}{\cA}
\newcommand{\PolyD}{{\tt Poly}_\dgb}
\newcommand{\rr}{A}
\newcommand{\Ew}{E}
\newcommand{\modl}{\cP}
\newcommand{\nmod}{P}
\newcommand{\Ic}{I}
\newcommand{\wt}{\delta}
\newcommand{\Sset}{\mathsf{S}}
\newcommand{\emb}{\iota}
\newcommand{\nS}{\mathsf{s}}
\newcommand{\Tset}{\mathsf{T}}
\newcommand{\nT}{\mathsf{t}}
\newcommand{\cls}{\mathsf{C}}
\newcommand{\cube}{\mathcal{N}}

\title[Landau--Siegel zeros and effective Brauer--Siegel]
{Landau--Siegel zeros over a base field\\ and effective
Brauer--Siegel}

\author{William D. Banks}
\address{Department of Mathematics\\
University of Missouri\\ Columbia, MO 65211\\ USA}
\email{bankswd@missouri.edu}


\date{\today}

\begin{document}

\begin{abstract}
Let $k$ be a number field. We prove that for every quadratic
extension $F/k$, every real zero $\beta\in(0,1)$ of $L(s,\chi_{F/k})$
satisfies $(1-\beta)\log d_F\ge c(k)$ with an effective constant
$c(k)>0$, and we deduce (through a reduction of Stark) that the
Brauer--Siegel theorem holds effectively and uniformly over all
Galois extensions $\Kf$ of $k$, in the form
\[
h_\Kf R_\Kf\ge c(k)\,d_\Kf^{1/2}/((2\pi)^{[\Kf:\Q]}\log d_\Kf).
\]
Stark's effective cases of the Brauer--Siegel theorem excluded the
fields with quadratic subfields, since those subfields carried the
ineffectivity of Siegel's theorem; over a base field $k$, the first
result removes this obstruction. Both rely on a uniform bound for
Chebotarev sums. Let $\Kf/k$ be a Galois extension with group
$\galKk$, and suppose $\Sset\subset\galKk$ contains the identity. For
any prime $\fp$ of $k$, let $\wt_\fp$ be the proportion of the primes
of $\Kf$ above $\fp$ whose Frobenius lies in $\Sset$. Then
\[
\sum_{\Nm\fp\le\XX}\wt_\fp\,\frac{\log\Nm\fp}{\Nm\fp}
\le\Big(\frac{|\Sset|}{|\galKk|}+\eps\Big)\log\XX+O(1)
\]
holds uniformly in $\Kf$ as soon as $\XX\ge d_\Kf^{\,\expC}$, with
$\expC$ and the implied constant depending only on $\eps$, $\expC$,
and the degree $[\Kf:\Q]$. The proof is the argument of OpenAI in
the form given by Kerr, simple vanishing at the points
$(\emb\theta)_{\emb\in\Sset}$, with $\theta$ in a box in $\cO_\Kf$,
and Masser's zero estimate. The field enters only through Masser's
invariant of the lattice of points, which we compute to be
$[\Kf:\Q]/|\Sset|$ by averaging over the Galois group of the normal
closure. Over $\Q$, the determinant method gives sharper Chebotarev
bounds, and we explain the comparison.
\end{abstract}

\maketitle

\setcounter{tocdepth}{1}
\tableofcontents

\section{Introduction and statement of results}
\label{sec:intro}

\subsection{Background}

Two preprints of October 2026 assert zero-free half-planes for every
number field. In~\cite{CT}, the Dedekind zeta function $\zeta_\Kf(s)$
of every number field $\Kf$ is shown to have no zeros with
$\mathrm{Re}\,s>31/32$, and the same half-plane is zero-free for the
$L$-function of every finite-order Hecke character of $\Kf$ where it
is holomorphic; the preprint states that its proofs were generated
and formalized in Lean by AxiomProver. In~\cite{Li}, the $L$-function
of every finite-order Hecke character over every number field is
zero-free for $\mathrm{Re}\,s>0.8749$. Both rest on recent preprints
of OpenAI cited there, and we have not verified them. If they are
accepted, they settle everything below about real zeros with
absolute constants, and more. A zero-free region of width $w$ for
every $\zeta_\Kf(s)$, Stark's lower bound for the residue
(Lemma~\ref{lem:residue}) when $\log d_\Kf\ge 1/(4w)$, and Hermite's
theorem below that, give the Brauer--Siegel theorem effectively for
every number field. The present paper is independent of these
results and unconditional. Its content is a short argument, with
the computation of Masser's invariant behind it, that excludes
Landau--Siegel zeros over a base field and makes the Brauer--Siegel
theorem effective for Galois extensions of that field, and it
claims nothing beyond that.

OpenAI~\cite{OAI} has recently shown that Dirichlet $L$-functions of
quadratic characters have no Landau--Siegel zeros via an argument in
the biquadratic field $\Q(\sqrt d,\sqrt 2\,)$; see~\cite{Ba} for an
overview of the argument with explicit constants and~\cite{Ke} for a
simplified adaptation, in which the vanishing to high order
of~\cite{OAI} is replaced by simple vanishing at many lattice points
and a zero estimate of Masser. In~\cite{Ke}, the auxiliary number
$\sqrt 2$ is there to raise the rank of a certain lattice to four
while retaining three variables, and this is the source of the
exponent $\tfrac34$ that arises there. Green~\cite{Gr} obtained the
sharp constant $\tfrac12$ by a different route, working exclusively
in $\Q(\sqrt d\,)$ and comparing the Archimedean and $p$-adic sizes
of a Gram determinant built from small integers of $\Q(\sqrt d\,)$
and the derivatives of monomials at them.
Kominers~\cite{Ko} observed that the determinant is not about
quadratic characters. Formed from the pairs $(\alpha,\sigma\alpha)$,
with $\alpha$ running over small integers of a Galois extension
$\Kf/\Q$ and $\sigma$ an automorphism, it bounds the logarithmic
mass of the primes whose Frobenius class is that of $\sigma$ by the
Chebotarev density of the class, uniformly in $\Kf$, with absolute
constants and no threshold on the length of the sum.

The present paper has two aims. The first is the Brauer--Siegel
theorem. For a number field $\Kf$ with class number $h_\Kf$,
regulator $R_\Kf$, and absolute discriminant $d_\Kf$, the theorem
asserts that $\log(h_\Kf R_\Kf)\sim\tfrac12\log d_\Kf$ as $\Kf$ runs
over any sequence of Galois extensions of $\Q$ for which
$[\Kf:\Q]/\log d_\Kf\to 0$, and its constants have long been
ineffective because the proof passes through Siegel's theorem on
$L(1,\chi)$.
Stark~\cite{St} located the ineffectivity. If $\Kf/k$ is a normal
extension and $\zeta_\Kf(s)$ has a real zero $\beta$ close to one,
then $\beta$ is a zero of $\zeta_F(s)$ for some field $F$ with
$k\subseteq F\subseteq\Kf$ and $[F:k]\le 2$. For a field built by a
tower of normal steps over $\Q$, an \emph{almost normal field} in
the terminology of later authors, the reduction continues down the
tower, and the only
sources of ineffectivity are the quadratic subfields; Stark's
effective cases are the almost normal fields without them. Now that
Landau--Siegel zeros of quadratic characters over $\Q$ are
excluded~\cite{OAI,Gr,Ko} with absolute constants, the theorem is
effective for every almost normal field, in particular for every
Galois extension of $\Q$. That much is immediate from Stark's
reduction, and we state it only in passing (Remark~\ref{rem:almost}).
Theorem~\ref{thm:bs} concerns the Galois extensions $\Kf$ of an
\emph{arbitrary} number field $k$. Such a $\Kf$ need not be almost
normal,
and Stark's reduction stops at $k$, where two obstructions remain,
the quadratic extensions $F/k$ and the zeta function of $k$ itself.
The first is removed by Theorem~\ref{thm:quadratic} and
Corollary~\ref{cor:siegel}, the exclusion of Landau--Siegel zeros
for the quadratic characters of $k$, and the second is a constant
depending only on $k$. The outcome is the Brauer--Siegel theorem,
effective and uniform over all Galois extensions of a fixed base
field.

The second aim is the method. The argument of~\cite{Ke}, carried
out in a Galois extension $\Kf/k$ of number fields with group
$\galKk$, with one variable for each element of a subset
$\Sset\subset\galKk$ containing the identity, bounds the logarithmic
mass of the primes of $k$ whose Frobenius lies in $\Sset$ by the
Chebotarev density $|\Sset|/|\galKk|$ of those primes, uniformly in
$\Kf$ (Theorem~\ref{thm:main}). The constant is the number of
variables $|\Sset|$ divided by the rank $[\Kf:\Q]$ of the lattice of
points, times $[k:\Q]$. In the quadratic case $\Kf=F(\rou)$, with
$\rou$ a primitive $\auxp$-th root of unity and $\Sset$ the identity
together with the coset of $\Gal(\Kf/F)$ occupied by the inert
primes, this is $\tfrac12+1/[\Kf:k]$, which can be made as close to
$\tfrac12$ as we please by letting $\auxp\to\infty$. The only input
from the field itself is Masser's invariant of the lattice, and its
computation, a Galois-theoretic argument specific to
$\Q(\sqrt d,\sqrt 2\,)$ in~\cite{Ke}, becomes in general an averaging
over the Galois group of the normal closure that uses only its
transitivity on embeddings (Proposition~\ref{prop:mu}).
Independently, and at the same time, Frei and Kai~\cite{FK} carried
the argument of~\cite{Ke} over an arbitrary number field $k$ and
obtained Corollary~\ref{cor:siegel} for real Hecke characters of
$k$, with a constant depending on $k$; their observation that a
single complex embedding suffices in the computation of Masser's
invariant is explained by Proposition~\ref{prop:mu}.
Over $\Q$ the determinant method, as used in~\cite{Ko}, gives sharper
bounds than Theorem~\ref{thm:main}, with the exact density, no
threshold in $\XX$, and no hypothesis on $\Sset$; the two methods are
compared in \S\ref{sec:remarks}.

\subsection{Results}

Throughout the paper, $k$ is a number field of degree
$\dk\defeq[k:\Q]$.
For a quadratic extension $F/k$, let $\chi=\chi_{F/k}$ be its
quadratic character, so that $\chi(\fp)=1,-1,0$ according to whether
$\fp$ splits, is inert, or ramifies in $F$, and
$L(s,\chi)=\zeta_F(s)/\zeta_k(s)$.

\begin{theorem}\label{thm:quadratic}
Let $k$ be a number field. For each $\eps>0$ there is a number
$\expC_0=\expC_0(\eps)$, independent of $k$, such that for every
$\expC\ge\expC_0$, every quadratic extension $F/k$, and
$\XX\defeq d_F^{\,\expC}$, we have
\[
\ssum{\Nm\fp\le \XX\\ \chi(\fp)=-1}\frac{\log\Nm\fp}{\Nm\fp}
\le\Big(\frac12+\eps\Big)\log \XX+O_{k,\eps,\expC}(1).
\]
\end{theorem}

The threshold $\expC_0$ is independent of $k$ because the field
$\Kf=F(\rou)$ in which the proof operates has a $\Q$-basis of integers
of
house about $d_F^{1/(2\dk)}$ (Lemma~\ref{lem:generator}), whereas a
general field with the same discriminant enjoys only Minkowski's
bound.
The fact that a bound of the shape of Theorem~\ref{thm:quadratic}
rules out
Landau--Siegel zeros is a standard piece of analytic number theory;
we include the deduction in \S\ref{sec:zeros} for completeness.

\begin{corollary}\label{cor:siegel}
For every number field $k$ there is an effectively computable
constant $c(k)>0$ such that for every quadratic extension $F/k$,
every real zero $\beta\in(0,1)$ of $L(s,\chi_{F/k})$ satisfies
$(1-\beta)\log d_F\ge c(k)$.
\end{corollary}

For $k=\Q$, this is the theorem of~\cite{OAI}; for general $k$, it
is also Theorem~1.1 of~\cite{FK}, stated there for real Hecke
characters, with a constant that is in principle effective. What is
added here is the threshold $\expC_0(\eps)$ of
Theorem~\ref{thm:quadratic}, which does not depend on $k$, the
argument that $c(k)$ is effectively computable, and
Theorem~\ref{thm:bs} below. The dependence of
$c(k)$ on $k$ enters through one analytic input only, Mertens'
theorem for $k$, and we comment on it in \S\ref{sec:uniform}.
Combined with Stark's reduction, Corollary~\ref{cor:siegel} yields
the effective Brauer--Siegel theorem over $k$.

For a number field $\Kf$, we write $h_\Kf$, $R_\Kf$, and $d_\Kf$ for
its class number,
regulator, and absolute discriminant~\cite[Chapter~I, \S\S5--7]{Ne}.

\begin{theorem}\label{thm:bs}
For every number field $k$, there is an effectively computable
constant $c(k)>0$ such that for every Galois extension $\Kf/k$ with
$\Kf\ne\Q$ the following hold.
\begin{enumerate}
\item[(i)] Every real zero $\beta\in(0,1)$ of $\zeta_\Kf(s)$
satisfies $(1-\beta)\log d_\Kf\ge c(k)$.
\item[(ii)] $h_\Kf R_\Kf\ge
c(k)\,d_\Kf^{1/2}\big/\big((2\pi)^{[\Kf:\Q]}\log d_\Kf\big)$.
\end{enumerate}
In particular,
\[
\log(h_\Kf R_\Kf)\ge\tfrac12\log d_\Kf-\log\log d_\Kf
-[\Kf:\Q]\log 2\pi-C_k,
\]
thus the Brauer--Siegel theorem holds effectively, with constants
depending only on $k$, over every sequence of Galois extensions
$\Kf$ of $k$ with $[\Kf:\Q]/\log d_\Kf\to 0$.
\end{theorem}

The proof is given in \S\ref{sec:zeros}. The matching upper bound
\[
\log(h_\Kf R_\Kf)\le\tfrac12\log d_\Kf+O([\Kf:\Q]\log\log d_\Kf)
\]
is classical and effective, for instance by Louboutin's bound
\[
\kappa_\Kf\le\big(\er\log d_\Kf/(2[\Kf:\Q]-2)\big)^{[\Kf:\Q]-1}
\]
for the residue~\cite[Theorem~1]{Lo}. Thus, the essential content of
Theorem~\ref{thm:bs} is the lower bound. Stark's effective cases of
the Brauer--Siegel theorem
required the field to have no quadratic subfield; Theorem~\ref{thm:bs}
removes this requirement for Galois extensions of $k$ at the cost of a
constant depending on $k$, and the dependence is through the real
zeros of $\zeta_k(s)$ alone (\S\ref{sec:uniform}).

Theorem~\ref{thm:quadratic} is the quadratic case of a uniform
upper bound for Chebotarev sums. To formulate it, we first introduce
some terminology. Let $\Kf/k$ be a Galois extension
of number fields with group $\galKk\defeq\Gal(\Kf/k)$, and put
\[
\dk\defeq[k:\Q],\qquad \dK\defeq[\Kf:\Q],\qquad |\galKk|=\dK/\dk .
\]
Let $\Sset\subset\galKk$ be a subset containing the identity. For a
prime $\fp$ of $k$ unramified in $\Kf$ and a prime $\fP$ of $\Kf$
above $\fp$, we use $\Frob_{\,\fP}$ to denote the Frobenius element in
$\galKk$.
The set $\cls_\fp\defeq\{\Frob_{\,\fP}:\fP\mid\fp\}$ is a conjugacy
class of $\galKk$. We put
\be\label{eq:weight}
\wt_\fp\defeq\frac{|\Sset\cap\cls_\fp|}{|\cls_\fp|}
\quad\text{($\fp$ unramified in $\Kf$)},\qquad
\wt_\fp\defeq 0\quad\text{($\fp$ ramified in $\Kf$)}.
\ee
Since $\galKk$ permutes the primes above $\fp$ transitively, and
$\Frob_{\,g\fP}=g\Frob_{\,\fP}g^{-1}$ for all $g\in\galKk$, the weight
$\wt_\fp$ is the proportion of the primes $\fP\mid\fp$ for which
$\Frob_{\,\fP}\in\Sset$. When $\Sset$ is a union of conjugacy classes,
$\wt_\fp$ is the indicator of $\cls_\fp\subset\Sset$. By the
Chebotarev density theorem, for fixed $\Kf$ and $\Sset$, we have
\be\label{eq:chebotarev}
\sum_{\Nm\fp\le\XX}\wt_\fp\,\frac{\log\Nm\fp}{\Nm\fp}
=\sum_{\cls}\frac{|\Sset\cap\cls|}{|\cls|}\cdot\frac{|\cls|}{|\galKk|}
\log\XX+O_\Kf(1)
=\frac{|\Sset|}{|\galKk|}\log\XX+O_\Kf(1),
\ee
where the second sum is taken over the conjugacy classes of $\galKk$.
The underlying result of the paper is an upper bound of this
strength that is \emph{uniform in $\Kf$}.

\begin{theorem}\label{thm:main}
For every $\eps>0$ and $\dK\ge 1$, there is a number
$\expC_0=\expC_0(\eps,\dK)$ such that the following holds. Let $\Kf/k$
be a Galois extension of number fields with $[\Kf:\Q]=\dK$ and
group~$\galKk$, let $\Sset\subset\galKk$ contain the identity, let
$\expC\ge\expC_0$ and $\XX\defeq d_\Kf^{\,\expC}$. Then
\[
\sum_{\Nm\fp\le\XX}\wt_\fp\,\frac{\log\Nm\fp}{\Nm\fp}
\le\Big(\frac{|\Sset|}{|\galKk|}+\eps\Big)\log\XX
+O_{\dK,\eps,\expC}(1).
\]
\end{theorem}

By \eqref{eq:chebotarev} the constant $|\Sset|/|\galKk|$ is best
possible. Theorem~\ref{thm:main} is a consequence of the following
more precise inequality, in which the dependence on $\Kf$ is isolated
in one quantity. Let $\lamK\ge 1$ be the least number such that
$\cO_\Kf$ contains $\dK$ elements, linearly independent over $\Q$, all
of whose conjugates have absolute value at most $\lamK$. Minkowski's
theorem gives $\lamK\le d_\Kf^{1/2}$ (see \eqref{eq:minkowski}
below).

\begin{theorem}\label{thm:precise}
Let $\Kf/k$, $\galKk$, $\Sset$, $\dK$, $\dk$ be as above and let
$\wgt\ge 2$ be an integer. There is $\XX_0=\XX_0(\dK,\wgt)$ such that
for all $\XX\ge\max\big(\XX_0,(\log d_\Kf)^\dK\big)$, we have
\be\label{eq:precise}
\sum_{\Nm\fp\le \XX}\wt_\fp\,\frac{\log\Nm\fp}{\Nm\fp}
\le\Big(1+\frac{|\Sset|-1}{\wgt}\Big)
\Big(\frac{|\Sset|}{|\galKk|}\log \XX+\dk\log\lamK\Big)
+O_{\dK,\wgt}(1).
\ee
\end{theorem}

The identity element must belong to $\Sset$ because the argument
spends one
variable on $\theta$ itself, to which the Frobenius congruence
$\theta^{\Nm\fp}\equiv\sigma(\theta)$ ties the variable
$\sigma(\theta)$. The hypothesis is a normalization, since the
argument with $\emb_0(\theta)$ as the distinguished variable bounds
the mass of $\Sset\emb_0^{-1}$ for any $\emb_0\in\Sset$; see
Remark~\ref{rem:compare}.

Kominers~\cite{Ko} proves, for $\Kf/\Q$ Galois of degree $\dK$ and
a conjugacy class $\cls$ of size $c$, that
\[
\ssum{p\le\XX\\ \cls_p=\cls}\frac{\log p}{p}
\le\frac{c}{\dK}\Big(\log\XX+\frac12\log d_\Kf\Big)+\log\log 3\XX+O(c)
\qquad(\XX\ge 2)
\]
holds with an absolute implied constant.
Summed over the classes meeting $\Sset$ with the weights
\eqref{eq:weight}, this result implies our Theorem~\ref{thm:main}
for $k=\Q$ with $\eps$ replaced by $1/(2\expC)$, no threshold, and no
hypothesis
on $\Sset$. Theorem~\ref{thm:precise} is the analogue produced by
the Masser route; its field term $\dk\log\lamK$ plays the part of
$\frac{c}{2\dK}\log d_\Kf$, and its factor $1+(|\Sset|-1)/\wgt$ is the
price of simple vanishing in $|\Sset|$ variables, which the
derivatives in the determinant method avoid. We return to the
comparison in Remark~\ref{rem:compare}.

\subsection{Outline of the proof}\label{sec:outline}

Theorem~\ref{thm:precise} is the heart of the paper;
Theorems~\ref{thm:main} and~\ref{thm:quadratic} follow from it by
choices of the parameters. Its proof compares two counts of
polynomials, and it has five ingredients.

\medskip\noindent\textbf{The points.}
Fix $\omega_1,\dots,\omega_\dK\in\cO_\Kf$, linearly independent over
$\Q$, of house at most $\lamK$. For an integer $\bxs$ and
$\vn\in\cube\defeq\{0,\dots,\bxs-1\}^{\dK}$ put
$\theta_\vn\defeq\sum_t n_t\omega_t$, and attach to $\theta_\vn$ the
point
\[
\theta_\vn^{\Sset}\defeq\big(\emb(\theta_\vn)\big)_{\emb\in\Sset}
\in\C^{\nS},
\]
one coordinate for each element of $\Sset$ read as an embedding of
$\Kf$. These $\bxs^{\dK}$ points lie in the lattice $\lat$ generated
by $\omega_1^{\Sset},\dots,\omega_\dK^{\Sset}$, which has rank $\dK$
in $\C^{\nS}$, and every coordinate has house at most
$\hB\defeq\dK\lamK\bxs$.

\medskip\noindent\textbf{The zero estimate.}
A nonzero polynomial in $\nS$ variables that vanishes at all
$\bxs^{\dK}$ points has degree at least $((\bxs-1)/\nS)^{\dK/\nS}$
(Corollary~\ref{cor:zero}). This is Masser's theorem, whose exponent
is his invariant $\mu(\lat)$, a minimum over the subspaces
$V\subset\C^{\nS}$ of $(\dK-\rk(\lat\cap V))/(\nS-\dim V)$, and the
only input from the field is the computation $\mu(\lat)=\dK/\nS$
(Proposition~\ref{prop:mu}). It amounts to the statement that any
$\nT$ of the $\dK$ embeddings of $\Kf$ see at least the fraction
$\nT/\dK$ of every rational subspace of $\Kf$, and it is proved by
averaging over the Galois group of the normal closure of $\Kf$. We
fix the degree bound $\dgb$ just below Masser's threshold, so that no
nonzero polynomial of degree at most $\dgb$ vanishes at all the
points.

\medskip\noindent\textbf{The Frobenius congruence.}
Let $\fp$ be a prime of $k$ unramified in $\Kf$ with $\Nm\fp\ge\wgt$,
and let $\sigma\in\Sset\cap\cls_\fp$. Then
$\theta^{\Nm\fp}\equiv\sigma(\theta)$ modulo the product
$\fQ_{\fp,\sigma}$ of the primes above $\fp$ with Frobenius $\sigma$,
an ideal of norm $\Nm\fp^{\,|\galKk|/|\cls_\fp|}$
(Lemma~\ref{lem:frob}). The point is that $\sigma(\theta_\vn)$ is
again a coordinate of $\theta_\vn^{\Sset}$, because $\sigma\in\Sset$;
this is where the hypothesis on $\Sset$ enters, and the norm of
$\fQ_{\fp,\sigma}$ is where the weight $\wt_\fp$ comes from. We work
in the module $\PolyD$ of polynomials over $\cO_\Kf$ in the variables
$z_\emb$, $\emb\in\Sset$, of weighted degree at most $\dgb$, where
$z_1$ has weight one and every other variable has weight $\wgt$. The
module has a basis of polynomials
$z_1^{v}(z_1^{\Nm\fp}-z_\sigma)^{b}\vz^{\va^*}$
(Lemma~\ref{lem:basis}), whose values at the points lie in
$\fQ_{\fp,\sigma}^{\,b}$ with $b=\lfloor\alpha_1/\Nm\fp\rfloor$
(Lemma~\ref{lem:vanish}); the weights guarantee that the
substitution $z_1^{\Nm\fp}\mapsto z_\sigma$ does not leave $\PolyD$.
Hence the vector of values of $f\in\PolyD$ at the points, reduced
modulo $\modl\defeq\prod\fQ_{\fp,\sigma}^{\,m_\fp}$, takes at most
$\Nm\modl^{\,\rr}\prod\Nm\fQ_{\fp,\sigma}^{-\Ew_\fp}$ values, where
$\rr$ is the rank of $\PolyD$ and $\Ew_\fp$ is the sum of
$\lfloor\alpha_1/\Nm\fp\rfloor$ over its basis monomials
(Lemma~\ref{lem:residues}).

\medskip\noindent\textbf{The pigeonhole.}
The polynomials $f\in\PolyD$ whose coefficients satisfy
$\house{c_{\va}}\hB^{|\va|}\le\hgt$ number at least
$(\hgt/(\dK\lamK))^{\dK\rr}\hB^{-\dK W}$, where $W$ is the sum of the
degrees of the basis monomials. When there are more such polynomials
than residue vectors, two of them agree modulo $\modl$ at every point,
and their difference $\poly$ takes at every point a value in $\modl$
of house less than $\Nm\modl^{1/\dK}$, which is zero. Taking
logarithms, this happens as soon as
\[
\frac1{\dk}\sum_{\fp}\wt_\fp\,\Ew_\fp\log\Nm\fp
>W\log\hB+\rr\log(4\dK\rr\lamK)
\]
(Lemma~\ref{lem:pigeon}), the sum over the admissible primes with
$\Nm\fp\le\dgb$. The weight arises here, as
\[
\sum_{\sigma\in\Sset\cap\cls_\fp}\log\Nm\fQ_{\fp,\sigma}
=|\galKk|\,\wt_\fp\log\Nm\fp .
\]

\medskip\noindent\textbf{The comparison.}
The zero estimate forbids $\poly$, so the inequality fails. Up to
lower order terms, $\rr$, $W$, and $\Ew_\fp$ are integrals over a
simplex, with $W=(1+(\nS-1)/\wgt)\Ic$ and $\Ew_\fp=\Ic/\Nm\fp$ for
$\Ic\defeq\dgb^{\nS+1}/((\nS+1)!\,\wgt^{\nS-1})$
(Lemma~\ref{lem:count}), and dividing by $\Ic$ leaves
\[
\sum_{\Nm\fp\le\dgb}\wt_\fp\,\frac{\log\Nm\fp}{\Nm\fp}
\le\Big(1+\frac{\nS-1}{\wgt}\Big)
\big(\dk\log\bxs+\dk\log\lamK\big)+O_{\dK,\wgt}(1)
\]
(Proposition~\ref{prop:main}). The lattice has rank $\dK$, but there
are only $\nS$ variables, so $\log\dgb=\frac{\dK}{\nS}\log\bxs+O(1)$
and $\dk\log\bxs=\frac{\nS}{|\galKk|}\log\dgb+O(1)$; this is the
Chebotarev density. The term $\dk\log\lamK$ is the price of the
field, at most $\frac{\dK}{2}\log d_\Kf$ by Minkowski's theorem and
at most $\frac12\log d_F+C_k$ in the quadratic case
(Lemma~\ref{lem:generator}), and $\eps\log\XX$ absorbs it once $\XX$
is a large power of the discriminant. The deductions of
Corollary~\ref{cor:siegel} and Theorem~\ref{thm:bs} are classical. A
real zero $\beta$ of
$L(s,\chi_{F/k})$ depletes the split primes (Lemma~\ref{lem:bias}),
Mertens' theorem for $k$ says that all primes carry mass $\log\XX$,
and Theorem~\ref{thm:quadratic} caps the inert primes at
$(\tfrac12+\eps)\log\XX$, which leaves no room unless
$(1-\beta)\log d_F\gg 1$. For $\Kf/k$ Galois, Stark's reduction
places a real zero of $\zeta_\Kf(s)$ near one on a field $F$ with
$k\subseteq F\subseteq\Kf$ and $[F:k]\le 2$, where
Corollary~\ref{cor:siegel} or the zeta function of $k$ bounds it
away from one, and Stark's lower bound for the residue of
$\zeta_\Kf(s)$ in terms of its real zeros gives Theorem~\ref{thm:bs}.

\medskip

The paper is organized as follows.
\begin{center}
\renewcommand{\arraystretch}{1.15}
\begin{tabular}{@{}ll@{}}
\S\ref{sec:quadratic}&
Theorem~\ref{thm:quadratic} from Theorem~\ref{thm:precise}\\
\S\ref{sec:zeros}&
Corollary~\ref{cor:siegel} and Theorem~\ref{thm:bs}\\
\S\ref{sec:lattice}&
the lattice of points and the zero estimate\\
\S\ref{sec:poly}&
the Frobenius congruence and the pigeonhole step\\
\S\ref{sec:proof}&
Theorems~\ref{thm:main} and~\ref{thm:precise}\\
\S\ref{sec:remarks}&
comparison with the determinant method
\end{tabular}
\end{center}

\subsection{Conventions}

The letter $p$ always denotes a rational prime, $\fp$ a prime of $k$,
and $\fP$ a prime of $\Kf$. We write $\cO_k$, $\cO_\Kf$ for the rings
of integers, $d_k$, $d_\Kf$ for the absolute discriminants, and $\Nm$
for the absolute norm of ideals, the ring being clear from context;
for $x\in\Kf^\times$ we put $\Nm x\defeq|N_{\Kf/\Q}(x)|$. Constants
$C_\dK$ depend only on $\dK$, constants $C_k$ only on $k$, and they
may change from line to line. We use standard asymptotic notation; a
subscript on $O$ or $\ll$ names the parameters on which the implied
constant depends. For a finite set $\cA$, we use $|\cA|$ to denote its
cardinality. We fix an embedding of $\Kf$ in $\C$ and regard $\Kf$ as
a subfield of $\C$; the elements of $\galKk$ are then embeddings
$\Kf\to\C$. For $\theta\in\Kf$ we write
\[
\house{\theta}\defeq\max_{\iota}|\iota(\theta)|,
\]
the maximum over the $\dK$ embeddings $\iota\colon\Kf\to\C$; for
$\theta\in k$ this is the maximum over the embeddings of $k$. Then
$\house{\cdot}$ is submultiplicative and subadditive, and any nonzero
$\theta\in\cO_\Kf$ has $\house{\theta}\ge 1$ since its norm is a
nonzero integer.

\section{Quadratic extensions}\label{sec:quadratic}

Theorem~\ref{thm:quadratic} is deduced here from
Theorem~\ref{thm:precise}, which is proved in \S\ref{sec:proof};
nothing in this section or the next depends on how it is proved.

Let $k$ be a number field of degree $\dk$, let $F/k$ be a quadratic
extension, let $\auxp$ be a prime, let $\rou$ be a primitive
$\auxp$-th root of unity, and put $\Kf\defeq F(\rou)$. The field
enters Theorem~\ref{thm:precise} through the term $\dk\log\lamK$,
which \eqref{eq:minkowski} bounds by $\tfrac12\dK\log d_\Kf$. The
threshold in Theorem~\ref{thm:quadratic} is independent of $k$
because $\Kf$ has a $\Q$-basis of house about $d_F^{1/(2\dk)}$, and
such a basis exists once $F$ is generated by the square root of an
integer of $k$ all of whose conjugates have size about $d_F^{1/\dk}$.
Over $\Q$ the fundamental discriminant is such an integer, with
$C_k=1$ below.

\begin{lemma}\label{lem:generator}
There is $\tau\in\cO_k$ with $F=k(\sqrt\tau)$ and
$|\iota(\tau)|\le C_kd_F^{1/\dk}$ for every embedding
$\iota\colon k\to\C$. Consequently $\lamK\le C_kd_F^{1/(2\dk)}$, that
is, $\dk\log\lamK\le\tfrac12\log d_F+C_k$.
\end{lemma}

\begin{proof}
Write $F=k(\sqrt{\tau_0})$ with $\tau_0\in\cO_k$, and factor
$(\tau_0)=\fa\fb^2$ with $\fa$ squarefree. A prime of $k$ dividing
$\fa$ has odd valuation at $\tau_0$, so it ramifies in $F$. Hence
$\fa\mid\fd_{F/k}$ by~\cite[Chapter~III, Corollary~(2.12)]{Ne}, and
$\Nm\fa\le\Nm\fd_{F/k}=d_F/d_k^{2}$
by~\cite[Chapter~III, Corollary~(2.10)]{Ne}.

By Minkowski's theorem~\cite[Chapter~I, Theorem~(6.3)]{Ne}, the ideal
class of $\fb$ contains an integral ideal $\fc$ of norm at most $C_k$.
Write $\fb=\gamma\fc$ with $\gamma\in k^\times$ and put
$\tau_1\defeq\tau_0\gamma^{-2}$. Then $(\tau_1)=\fa\fc^2$, so
$\tau_1\in\cO_k$ and $\Nm\tau_1\le C_kd_F$, and $F=k(\sqrt{\tau_1})$.

Let $\iota$ run over the $\dk$ embeddings of $k$. By Dirichlet's unit
theorem~\cite[Chapter~I, Theorems~(7.3) and~(7.4)]{Ne}, the image of
$\cO_k^\times$ under $u\mapsto(\log|\iota(u)|)_\iota$ is a lattice of
full rank in the subspace of $\R^{\dk}$ of vectors with coordinate
sum zero and equal coordinates at conjugate embeddings. The vector
$\big(\log|\iota(\tau_1)|-\tfrac1{\dk}\log\Nm\tau_1\big)_\iota$ lies
in that subspace, so there is a unit $u$ with
\[
\Big|2\log|\iota(u)|+\log|\iota(\tau_1)|-\frac1{\dk}\log\Nm\tau_1\Big|
\le C_k
\]
for every $\iota$. Put $\tau\defeq\tau_1u^2$. Then $\tau\in\cO_k$,
$F=k(\sqrt\tau)$, and
$|\iota(\tau)|\le C_k(\Nm\tau_1)^{1/\dk}\le C_kd_F^{1/\dk}$ for every
$\iota$.

For the second claim, fix a $\Z$-basis of $\cO_k$ and let $\omega$ run
over it. Since $\Kf=k(\sqrt\tau,\rou)$, the elements
$\omega\sqrt\tau^{\,a}\rou^{j}$ with $a\in\{0,1\}$ and
$0\le j\le\auxp-2$ lie in $\cO_\Kf$ and span $\Kf$ over $\Q$, so some
$\dK$ of them form a $\Q$-basis of $\Kf$. Every conjugate of $\rou$
has absolute value one, every conjugate of $\sqrt\tau$ has absolute
value $|\iota(\tau)|^{1/2}\le C_kd_F^{1/(2\dk)}$ for some $\iota$, and
the $\omega$ have house at most $C_k$. Hence
$\lamK\le C_kd_F^{1/(2\dk)}$.
\end{proof}

\begin{proof}[Proof of Theorem~\ref{thm:quadratic}]
Let $\eps\in(0,1]$, put $\expC_0\defeq\lceil 4/\eps\rceil$, and fix a
prime $\auxp$ with $\auxp-1\ge 8\dk/\eps$. Let $F/k$ be quadratic,
and put $\Kf\defeq F(\rou)$, $\galKk\defeq\Gal(\Kf/k)$, and
$\galKk_F\defeq\Gal(\Kf/F)$, a subgroup of index $2$. Then
$\dK\le 2\dk(\auxp-1)$, and
$|\galKk|=[\Kf:k]\ge[k(\rou):k]\ge(\auxp-1)/\dk$, since
$k(\rou)\supset\Q(\rou)$.

Put $\Sset\defeq\{1\}\cup(\galKk\setminus\galKk_F)$. Then
\[
\nS=\frac{|\galKk|}{2}+1\mand
\frac{\nS}{|\galKk|}=\frac12+\frac1{|\galKk|}\le\frac12+\frac{\eps}{8}.
\]
For a prime $\fp$ of $k$ unramified in $\Kf$ and $\fP\mid\fp$, the
restriction of $\Frob_{\,\fP}$ to $F$ is the Frobenius of $\fP\cap F$
over $\fp$, which is nontrivial exactly when $\fp$ is inert in $F$.
Thus $\cls_\fp\subset\galKk\setminus\galKk_F$ when $\chi(\fp)=-1$,
and $\wt_\fp=1$ for every inert prime unramified in $\Kf$. An inert
prime ramified in $\Kf$ is ramified in $\Kf/F$ and so divides
$\auxp$; there are at most $\dk$ such primes, and each contributes at
most $\log\Nm\fp/\Nm\fp\le 1$.

Let $\expC\ge\expC_0$, $\XX\defeq d_F^{\,\expC}$, and
$\wgt\defeq\lceil 16\dk(\auxp-1)/\eps\rceil$. We have
$d_\Kf=d_F^{[\Kf:F]}\Nm\fd_{\Kf/F}$, and $\fd_{\Kf/F}$ divides the
discriminant of the minimal polynomial of $\rou$ over $F$, which
divides $\disc(\Phi_\auxp)=\pm\auxp^{\auxp-2}$; hence
\[
\log d_\Kf\le(\auxp-1)\log d_F+2\dk(\auxp-2)\log\auxp .
\]
Since $\dK$ takes finitely many values, the conditions
\[
\XX\ge\XX_0(\dK,\wgt)\mand\XX\ge(\log d_\Kf)^{\dK}
\]
of Theorem~\ref{thm:precise} hold once $d_F$ is large in terms of
$k$, $\auxp$, $\wgt$, and $\expC$; assume that this is the case. By
Lemma~\ref{lem:generator} and $\log d_F=\frac{1}{\expC}\log\XX$, we
have $\dk\log\lamK\le\frac1{2\expC}\log\XX+C_k$, and
Theorem~\ref{thm:precise} gives
\begin{align*}
\ssum{\Nm\fp\le \XX\\ \chi(\fp)=-1}\frac{\log\Nm\fp}{\Nm\fp}
&\le\sum_{\Nm\fp\le\XX}\wt_\fp\frac{\log\Nm\fp}{\Nm\fp}+\dk\\
&\le\Big(1+\frac{\nS-1}{\wgt}\Big)
\Big(\frac12+\frac{\eps}{8}+\frac1{2\expC}\Big)\log\XX
+O_{k,\auxp,\wgt,\expC}(1).
\end{align*}
Here $(\nS-1)/\wgt\le\dk(\auxp-1)/\wgt\le\eps/16$ by the choice of
$\wgt$, and $1/(2\expC)\le\eps/8$ by the choice of $\expC_0$, so the
coefficient of $\log\XX$ is at most
\[
\Big(1+\frac{\eps}{16}\Big)\Big(\frac12+\frac{\eps}{4}\Big)
<\frac12+\eps .
\]
Since $\auxp$ and $\wgt$ depend only on $k$ and $\eps$, the error
term is $O_{k,\eps,\expC}(1)$, and $\expC_0$ depends on $\eps$ alone.
The remaining fields $F$ have bounded discriminant, so there are
finitely many of them by Hermite's theorem, and for each of them the
bound holds with the constant $\tfrac12$ by \eqref{eq:chebotarev}.
\end{proof}

More precisely, the proof gives, for every prime $\auxp$, every
integer $\wgt\ge 2$, every $\expC\ge 2$, and $d_F$ large in terms of
$k,\auxp,\wgt,\expC$,
\be
\ssum{\Nm\fp\le \XX\\ \chi(\fp)=-1}\frac{\log\Nm\fp}{\Nm\fp}
\le\Big(1+\frac{|\galKk|}{2\wgt}\Big)
\Big(\frac12+\frac1{|\galKk|}+\frac{1}{2\expC}\Big)\log \XX
+O_{k,\auxp,\wgt,\expC}(1),
\ee
with $|\galKk|=2(\auxp-1)$ whenever $\auxp\nmid d_k$ and
$F\ne k(\sqrt{\auxp^*})$, $\auxp^*\defeq(-1)^{(\auxp-1)/2}\auxp$; the
case $\auxp=3$ over $\Q$ is the argument of~\cite{Ke} with $\sqrt{-3}$
in place of $\sqrt 2$, and gives $\tfrac34+\eps$.

\section{Landau--Siegel zeros and the Brauer--Siegel theorem}
\label{sec:zeros}

\begin{lemma}\label{lem:bias}
Let $\beta\in(0,1)$ be a real zero of $L(s,\chi)$ and $\XX\ge 3$. Then
\[
\ssum{\Nm\fp\le \XX\\ \chi(\fp)=1}\frac{\log\Nm\fp}{\Nm\fp}
\ll\log d_F+(1-\beta)(\log \XX)^2,
\]
with an absolute implied constant.
\end{lemma}

\begin{proof}
Put $s_0\defeq 1+1/\log\XX$, so that $1<s_0\le 2$. The Dirichlet
series
\[
-\frac{\zeta_F'}{\zeta_F}(s_0)=\sum_{\fa}\Lambda_F(\fa)\Nm\fa^{-s_0},
\]
the sum over the nonzero ideals of $F$, has nonnegative
coefficients. A prime $\fp$ of $k$ with $\chi(\fp)=1$ and
$\Nm\fp\le\XX$ has two primes of $F$ above it, each of norm
$\Nm\fp$, and $\Nm\fp^{-1/\log\XX}\ge\er^{-1}$; hence
\[
\ssum{\Nm\fp\le\XX\\ \chi(\fp)=1}\frac{\log\Nm\fp}{\Nm\fp}
\le\frac{\er}{2}\Big(-\frac{\zeta_F'}{\zeta_F}(s_0)\Big).
\]
On the other hand, $\beta$ is a zero of
$\zeta_F(s)=\zeta_k(s)L(s,\chi)$,
and the Hadamard product of the completed zeta function of $F$
gives, for real $s_0$ with $1<s_0\le 2$,
\[
-\frac{\zeta_F'}{\zeta_F}(s_0)
\le\frac1{s_0-1}+\frac12\log d_F+O([F:\Q])-\frac1{s_0-\beta},
\]
with an absolute implied constant, every other nontrivial zero
contributing a negative term; see~\cite[Lemma~3, (9) and (10)]{St}
or~\cite[proof of Lemma~8.1]{LO}, and~\cite[Chapter~14]{Da} for
$k=\Q$. Here
$[F:\Q]\ll\log d_F$ by Minkowski's bound, and
\[
\frac1{s_0-1}-\frac1{s_0-\beta}
=\frac{1-\beta}{(s_0-1)(s_0-\beta)}
\le\frac{1-\beta}{(s_0-1)^2}=(1-\beta)(\log\XX)^2,
\]
since $s_0-\beta\ge s_0-1$. The lemma follows.
\end{proof}

We also need Mertens' theorem for $k$, which asserts
\be\label{eq:mertens}
\sum_{\Nm\fp\le x}\frac{\log\Nm\fp}{\Nm\fp}=\log x+O_k(1)
\qquad(x\ge 2);
\ee
see Rosen~\cite{Ro}, and Garcia and Lee~\cite[Theorem~1]{GL} for an
explicit constant in terms of $\dk$, $d_k$, and the residue
$\kappa_k$ of $\zeta_k(s)$ at $s=1$, for $\dk\ge 2$. We also use the
trivial bound for the ramified primes,
\be\label{eq:ramified}
\sum_{\chi(\fp)=0}\frac{\log\Nm\fp}{\Nm\fp}
\le\frac{\omega(\fd_{F/k})}{\er}\le\log d_F,
\ee
where $\omega(\fd_{F/k})\le\log_2\Nm\fd_{F/k}$ is the number of primes
dividing $\fd_{F/k}$ and $(\log t)/t\le 1/\er$ for $t>0$.

\begin{proof}[Proof of Corollary~\ref{cor:siegel}]
Let $c_1$ be the implied constant in Lemma~\ref{lem:bias}, an
absolute constant, and apply Theorem~\ref{thm:quadratic} with
$\eps=\tfrac18$ and $\expC\defeq\max(\expC_0(\tfrac18),8c_1,8)$, so
that $\XX=d_F^{\,\expC}$ and
$\log d_F=\frac{1}{\expC}\log\XX\le\tfrac18\log\XX$.
By \eqref{eq:mertens}, Theorem~\ref{thm:quadratic}, and
\eqref{eq:ramified}, the split primes satisfy
\[
\ssum{\Nm\fp\le\XX\\ \chi(\fp)=1}\frac{\log\Nm\fp}{\Nm\fp}
\ge\log\XX-\Big(\frac12+\frac18\Big)\log\XX-\frac18\log\XX-O_k(1)
=\frac14\log\XX-O_k(1).
\]
Let $\beta\in(0,1)$ be a real zero of $L(s,\chi)$. By
Lemma~\ref{lem:bias} and $c_1\log d_F\le\tfrac18\log\XX$, the same
sum is at most $c_1(1-\beta)(\log\XX)^2+\tfrac18\log\XX$. Hence
\[
c_1(1-\beta)(\log\XX)^2\ge\tfrac18\log\XX-O_k(1)\ge\tfrac1{16}\log\XX
\]
once $d_F$ is large in terms of $k$, and then
$(1-\beta)\log d_F\ge 1/(16c_1\expC)$. The remaining fields $F$
have bounded discriminant, so there are finitely many of them by
Hermite's theorem. For each of them, $L(s,\chi)$ has finitely many
real zeros in $(0,1)$, and the largest of them, $\beta_F$, satisfies
$\beta_F<1$ because $L(1,\chi)\ne 0$. The corollary follows on
decreasing $c(k)$ to $\min_F(1-\beta_F)\log d_F$.

Every constant is effectively computable. The constants of
Theorem~\ref{thm:quadratic} and Lemma~\ref{lem:bias} and Mertens'
constant for $k$ are explicit, Hermite's bound is effective, and for
each remaining $F$ with $\beta_F\ge\tfrac12$, a positive lower bound
for $1-\beta_F$ follows from the mean value theorem,
$L(1,\chi)=L'(s_1,\chi)(1-\beta_F)$ for some $s_1\in(\beta_F,1)$,
since $L(1,\chi)=\kappa_F/\kappa_k>0$ is computable from the class
number formula, and $L'(\sigma,\chi)$ is bounded on $[\tfrac12,1]$ by
an explicit function of $d_F$ and $\dk$, by Rademacher's explicit
convexity bound~\cite[Theorem~5]{Ra}.
\end{proof}

\subsection{The Brauer--Siegel theorem over $k$}

For a number field $\Kf$, let $h_\Kf$, $R_\Kf$, and $w_\Kf$ be
its class number, regulator, and number of roots of unity, let $r_1$
and $r_2$ be its numbers of real and complex places, and let
\[
\kappa_\Kf\defeq\lim_{s\to 1}(s-1)\zeta_\Kf(s)
=\frac{2^{r_1}(2\pi)^{r_2}h_\Kf R_\Kf}{w_\Kf\,d_\Kf^{1/2}},
\]
the residue of $\zeta_\Kf(s)$ at $s=1$~\cite[Chapter~VII,
Corollary~(5.11)]{Ne}. We use two results of Stark. The first is
his generalization of a theorem of Heilbronn, which reduces a real
zero near one to a subextension of degree at most two.

\begin{lemma}[{Stark~\cite[Lemma~3 and Theorem~3]{St}}]
\label{lem:stark}
Let $\Kf/k$ be a Galois extension with $\Kf\ne\Q$, and let $\beta$ be
a real zero of $\zeta_\Kf(s)$ with $1-\beta\le 1/(4\log d_\Kf)$. Then
$\beta$ is a simple zero, and there is a field $F$ with
$k\subseteq F\subseteq\Kf$ and $[F:k]\le 2$ such that
$\zeta_F(\beta)=0$.
\end{lemma}

Stark's Lemma~3 says that $\zeta_\Kf(s)$ has at most one zero in
the region $\sigma\ge 1-1/(4\log d_\Kf)$,
$|t|\le 1/(4\log d_\Kf)$, and that it is real and simple; his
Theorem~3 says that a simple real zero of $\zeta_\Kf(s)$, for
$\Kf/k$ normal, is a zero of $\zeta_F(s)$ for $F=k$ or $F$ quadratic
over $k$.

The second result is his lower bound for the residue in terms of the
real zeros. Let $\beta_\Kf$ be the largest real zero of
$\zeta_\Kf(s)$ in $[1-1/(4\log d_\Kf),1)$ if there is one, and
$\beta_\Kf\defeq 1-1/(4\log d_\Kf)$ otherwise.

\begin{lemma}[{Stark~\cite[Lemma~4]{St}}]\label{lem:residue}
There is an absolute effective constant $c_2>0$ such that
$\kappa_\Kf\ge c_2(1-\beta_\Kf)$ for every number field $\Kf\ne\Q$.
\end{lemma}

\medskip

Now we come to the proof of Theorem~\ref{thm:bs}, which traces a
real zero $\beta$ of $\zeta_\Kf(s)$ down the tower until it meets
one of the two inputs.
\[
\begin{tikzcd}[row sep=large, column sep=tiny]
& \Kf
  \arrow[d, "\text{Lemma~\ref{lem:stark}}"',
    "{\zeta_\Kf(\beta)=0,\ 1-\beta\le 1/(4\log d_\Kf)}"]
&\\
& F \arrow[dl, "{\zeta_k(\beta)=0}"']
    \arrow[dr, "{L(\beta,\chi_{F/k})=0}"]
  & \\
{(1-\beta)\log d_k\ge\eta_k}
& & {(1-\beta)\log d_F\ge c(k)}
\end{tikzcd}
\]
Here $k\subseteq F\subseteq\Kf$ with $[F:k]\le 2$, and
$\zeta_F(s)=\zeta_k(s)L(s,\chi_{F/k})$ when $[F:k]=2$; the left
branch is the only one when $F=k$, and the right branch is
Corollary~\ref{cor:siegel}. The constant $\eta_k$ on the left is
the contribution of the base field,
\[
\eta_k\defeq\min\big\{(1-\beta')\log d_k:\beta'\in(0,1),\
\zeta_k(\beta')=0\big\},
\]
with $\eta_k\defeq 1$ if $\zeta_k(s)$ has no real zero in $(0,1)$,
as for $k=\Q$. Since $\zeta_k(s)$ has finitely many zeros in $(0,1)$
and $\zeta_k(1)\ne 0$, we have $\eta_k>0$.

\begin{proof}[Proof of Theorem~\ref{thm:bs}]
Let $c(k)$ be the constant of Corollary~\ref{cor:siegel}, and let
$c_2$ be the constant of Lemma~\ref{lem:residue}, which we may take
to be at most one.

(i) Let $\Kf/k$ be Galois with $\Kf\ne\Q$, and let $\beta\in(0,1)$
be a real zero of $\zeta_\Kf(s)$. We prove
\be\label{eq:three}
(1-\beta)\log d_\Kf\ge\min\big(\tfrac14,\,c(k),\,\eta_k\big).
\ee
We use $d_k\le d_F\le d_\Kf$ for every
$F$ between $k$ and $\Kf$, which holds since the discriminant of an
extension is divisible by a power of the discriminant of the base.

\emph{Case 1: $1-\beta>1/(4\log d_\Kf)$.} Then
$(1-\beta)\log d_\Kf>\tfrac14$.

Otherwise $1-\beta\le 1/(4\log d_\Kf)$, and Lemma~\ref{lem:stark}
provides $F$ with $k\subseteq F\subseteq\Kf$, $[F:k]\le 2$, and
$\zeta_F(\beta)=0$. If $F=k$, then $\zeta_k(\beta)=0$; if $[F:k]=2$,
then $\zeta_F(s)=\zeta_k(s)L(s,\chi_{F/k})$, so $\beta$ is a zero of
one of the two factors.

\emph{Case 2: $\zeta_k(\beta)=0$.} Then
$(1-\beta)\log d_\Kf\ge(1-\beta)\log d_k\ge\eta_k$ by the definition
of $\eta_k$.

\emph{Case 3: $[F:k]=2$ and $L(\beta,\chi_{F/k})=0$.} Then
$(1-\beta)\log d_\Kf\ge(1-\beta)\log d_F\ge c(k)$ by
Corollary~\ref{cor:siegel}.

This proves \eqref{eq:three}, and (i) holds with
$c(k)$ replaced by $c_3(k)\defeq c_2\min(\tfrac14,c(k),\eta_k)$.

(ii) Let $\beta_\Kf$ be as in Lemma~\ref{lem:residue}. Either
$\beta_\Kf$ is a real zero of $\zeta_\Kf(s)$, and then
\eqref{eq:three} applies to it, or $\beta_\Kf=1-1/(4\log d_\Kf)$;
in both cases $1-\beta_\Kf\ge\min(\tfrac14,c(k),\eta_k)/\log d_\Kf$.
By Lemma~\ref{lem:residue} and the residue formula,
\[
h_\Kf R_\Kf=\frac{w_\Kf\,d_\Kf^{1/2}}{2^{r_1}(2\pi)^{r_2}}\,\kappa_\Kf
\ge\frac{d_\Kf^{1/2}}{(2\pi)^{\dK}}\cdot c_2(1-\beta_\Kf)
\ge\frac{c_3(k)\,d_\Kf^{1/2}}{(2\pi)^{\dK}\log d_\Kf},
\]
since $w_\Kf\ge 1$ and $2^{r_1}(2\pi)^{r_2}\le(2\pi)^{\dK}$. This is
(ii) with the constant $c_3(k)$, and taking logarithms gives the
displayed inequality of the theorem with $C_k\defeq-\log c_3(k)$.

Finally, $c_3(k)$ is effectively computable. The constants $c(k)$
and $c_2$ are, by Corollary~\ref{cor:siegel} and Stark's proof of
Lemma~\ref{lem:residue}, and so is $\eta_k$. A real zero
$\beta'<\tfrac12$ of $\zeta_k(s)$ contributes at least
$\tfrac12\log d_k$, and for $\beta'\ge\tfrac12$ the entire function
$(s-1)\zeta_k(s)$ vanishes at $\beta'$ and takes the value
$\kappa_k>0$ at $s=1$, so $1-\beta'$ is bounded below as in the
proof of Corollary~\ref{cor:siegel}.
\end{proof}

\begin{remark}\label{rem:almost}
If $k$ is almost normal over $\Q$, every Galois extension
$\Kf$ of $k$ is as well, and Theorem~\ref{thm:bs} holds with an
absolute
constant in place of $c(k)$. Indeed, Stark's Lemma~10~\cite{St}, the
tower form of Lemma~\ref{lem:stark}, places a real zero $\beta$ of
$\zeta_\Kf(s)$ with $1-\beta\le 1/(16\log d_\Kf)$ on a subfield
$F\subset\Kf$ with $[F:\Q]\le 2$, where the theorem
over~$\Q$~\cite{OAI,Gr,Ko} gives $(1-\beta)\log d_F\ge c_0$ with an
absolute constant; the rest of the proof is unchanged.
\end{remark}

\subsection{What is uniform in the field}\label{sec:uniform}

The proof of \eqref{eq:precise} is uniform in $\Kf$ and $k$ except
for the term $\dk\log\lamK$, which is at most
$\frac{\dK}{2}\log d_\Kf$ in general and $\frac12\log d_F+C_k$ in the
quadratic case, and which $\eps\log\XX$ absorbs once
$\XX\ge d_\Kf^{\,\expC_0}$. The one input to
Corollary~\ref{cor:siegel} that is not uniform is Mertens' constant
in \eqref{eq:mertens}. By
the explicit formula for $\zeta_k(s)$, that constant is governed by
the real zeros of $\zeta_k(s)$ near $s=1$, the phenomenon under study
one degree down. The same zeros enter Theorem~\ref{thm:bs} a
second time, through $\eta_k$ in the proof. Thus
Theorem~\ref{thm:quadratic} is as uniform as the house bound, while
its consequences for the zeros and for $h_\Kf R_\Kf$ are uniform
only over extensions of a fixed $k$; the obstruction is the
exceptional zero of $\zeta_k(s)$ for a field $k$ that is not almost
normal, and we do not address it.

\section{The lattice and the zero estimate}\label{sec:lattice}

\subsection{Integers of small house}

Recall that $\dK=[\Kf:\Q]$, and let $r_1$ and $r_2$ be the numbers of
real and complex places of $\Kf$, so that $\dK=r_1+2r_2$. The quantity
$\lamK$ is controlled by the discriminant. Under the Minkowski
embedding, $\cO_\Kf$ is a lattice in $\R^{r_1}\times\C^{r_2}$ of
covolume $2^{-r_2}d_\Kf^{1/2}$, and $\house{\cdot}$ is the maximum
norm on this space, whose unit ball has volume $2^{r_1}\pi^{r_2}$. Let
$\lam_1\le\dots\le\lam_\dK$ be the successive minima of $\cO_\Kf$ with
respect to $\house{\cdot}$. By Minkowski's theorem on successive
minima~\cite[Chapter~VIII, Theorem~V]{Ca}, we have
\[
\lam_1\cdots\lam_\dK\le(2/\pi)^{r_2}d_\Kf^{1/2}\le d_\Kf^{1/2}.
\]
A nonzero algebraic integer has house at least one, so $\lam_1\ge 1$,
and therefore
\be\label{eq:minkowski}
\lamK\le\lam_\dK\le d_\Kf^{1/2}.
\ee

From now on, we fix $\omega_1,\dots,\omega_\dK\in\cO_\Kf$, linearly
independent over $\Q$, with
\be\label{eq:lam}
\house{\omega_t}\le\lamK\qquad(1\le t\le \dK).
\ee
Such elements exist by the definition of $\lamK$; they need not form a
$\Z$-basis of $\cO_\Kf$.

\subsection{The points}

Fix an integer $\bxs\ge 2$ and put
\[
\cube\defeq\{0,1,\dots,\bxs-1\}^{\dK},\qquad
\hB\defeq\dK\lamK\bxs.
\]
For $\vn=(n_1,\dots,n_\dK)\in\cube$ we define
\[
\theta_\vn\defeq n_1\omega_1+\dots+n_\dK\omega_\dK\in\cO_\Kf,
\]
so that $\house{\theta_\vn}\le\hB$ by \eqref{eq:lam}. Recall that
$\nS\defeq|\Sset|$. For $\theta\in\Kf$ and a set $\Tset$ of embeddings
$\Kf\to\C$ we write
\be
\theta^{\Tset}\defeq\big(\iota(\theta)\big)_{\iota\in\Tset}
\in\C^{|\Tset|},
\ee
and for $\vu=(u_1,\dots,u_r)\in\Kf^r$ we write $\vu^{\Tset}$ for the
$|\Tset|\times r$ matrix with rows $\iota(\vu)$, $\iota\in\Tset$,
where $\iota(\vu)\defeq\big(\iota(u_1),\dots,\iota(u_r)\big)$; that
is, its columns are $u_1^{\Tset},\dots,u_r^{\Tset}$. The map
$\theta\mapsto\theta^{\Tset}$ is $\Q$-linear, and $\theta^{\Tset}$
determines $\theta$ when $\Tset\ne\emptyset$, since each $\iota$ is
injective. Our points are
\[
\theta_\vn^{\Sset}\in\cO_\Kf^{\,\nS}\subset\C^{\nS}\qquad(\vn\in\cube),
\]
every coordinate of which has house at most $\hB$, and by linearity
\[
\theta_\vn^{\Sset}=n_1\omega_1^{\Sset}+\dots+n_\dK\omega_\dK^{\Sset}.
\]
Let $\lat\subset\C^{\nS}$ be the subgroup generated by
$\omega_1^{\Sset},\dots,\omega_\dK^{\Sset}$, and let
$\pi\colon\C^{\dK}\to\C^{\nS}$ be the linear map with
$\pi(\ve_t)=\omega_t^{\Sset}$, so that $\lat=\pi(\Z^{\dK})$. The
objects introduced so far fit into the commutative diagram
\[
\begin{tikzcd}[row sep=large, column sep=large]
\cube \arrow[r, hook] \arrow[d, "\vn\mapsto\theta_\vn"']
& \Z^{\dK} \arrow[r, hook] \arrow[d, "\vv\mapsto\sum_t v_t\omega_t"']
& \Q^{\dK} \arrow[d, "\vv\mapsto\sum_t v_t\omega_t"']
\arrow[dd, bend left=70, "\pi", start anchor=east, end anchor=east] \\
\{\theta_\vn\} \arrow[r, hook]
\arrow[d, "\theta\mapsto\theta^{\Sset}"']
& \cO_\Kf \arrow[r, hook]
\arrow[d, "\theta\mapsto\theta^{\Sset}"']
& \Kf \arrow[d, "\theta\mapsto\theta^{\Sset}"'] \\
\{\theta_\vn^{\Sset}\} \arrow[r, hook] & \lat \arrow[r, hook]
& \C^{\nS}
\end{tikzcd}
\]
in which the horizontal arrows are inclusions and the vertical
composites are restrictions of $\pi$. For $\vv\in\Q^{\dK}$ we have
$\pi(\vv)=(v_1\omega_1+\dots+v_\dK\omega_\dK)^{\Sset}$, which vanishes
only if $\vv=\vzero$, since $\theta^{\Sset}$ determines $\theta$ and
the $\omega_t$ are linearly independent over $\Q$. Hence $\pi$ is
injective on $\Q^{\dK}$, and
$\omega_1^{\Sset},\dots,\omega_\dK^{\Sset}$ is a basis of $\lat$,
which has rank $\dK$.

\subsection{The zero estimate}

Let $\kappa\ge 1$, and let $\lat\subset\C^{\kappa}$ be a nontrivial
finitely generated subgroup of rank $l$. For a linear subspace
$V\subset\C^{\kappa}$ the group $\lat\cap V$ has rank at most $l$, and
we define
\[
\mu(\lat)\defeq\min_{V}\frac{l-\rk(\lat\cap V)}{\kappa-\dim V},
\]
the minimum taken over all \emph{proper} linear subspaces $V$ of
$\C^{\kappa}$, including $V=\{\vzero\}$ (and thus,
$\mu(\lat)\le l/\kappa$). This is Masser's generalized Dirichlet
exponent, and we use his zero estimate in the following form.

\begin{theorem}[{Masser~\cite[Theorem~1]{Ma}}]\label{thm:masser}
Let $\lat\subset\C^{\kappa}$ be a nontrivial subgroup of rank $l$ with
basis $\vv_1,\dots,\vv_l$, and let $M\ge 1$. If a nonzero polynomial
$\poly\in\C[z_1,\dots,z_\kappa]$ vanishes at every point
$m_1\vv_1+\dots+m_l\vv_l$ with integers $0\le m_j\le M$, then
\[
\deg \poly\ge(M/\kappa)^{\mu(\lat)}.
\]
\end{theorem}

The field enters the zero estimate only through the next
proposition, which
asserts that any $\nT$ of the $\dK$ embeddings of $\Kf$ see at least
the fraction $\nT/\dK$ of every rational subspace.

\begin{proposition}\label{prop:mu}
$\mu(\lat)\ge\dK/\nS$.
\end{proposition}

\begin{proof}
Let $V\subset\C^{\nS}$ be a proper subspace of dimension $d$, and put
$r\defeq\rk(\lat\cap V)$. We must show that $r\le\dK d/\nS$, which is
equivalent to $(\dK-r)/(\nS-d)\ge\dK/\nS$.

Since $\pi$ is injective on $\Q^{\dK}$, the set
\[
U\defeq\{v_1\omega_1+\dots+v_\dK\omega_\dK:\vv\in\Q^{\dK},\
\pi(\vv)\in V\}
\]
is a $\Q$-subspace of $\Kf$ of dimension $r$; in other words, $U$ is
the image of $\pi^{-1}(V)\cap\Q^{\dK}$ in the diagram
\[
\begin{tikzcd}[column sep=large, row sep=large]
\pi^{-1}(V)\cap\Q^{\dK} \arrow[r, "\sim"'] \arrow[d, hook]
\arrow[rr, bend left=20, "\pi", start anchor=north, end anchor=north]
& U \arrow[r, "u\,\mapsto u^{\Sset}"'] \arrow[d, hook]
& V \arrow[d, hook] \\
\Q^{\dK} \arrow[r, "\sim"']
\arrow[rr, bend right=40, "\pi"', start anchor=south, end anchor=south]
& \Kf \arrow[r, "\theta\,\mapsto\theta^{\Sset}"']
& \C^{\nS}
\end{tikzcd}
\]
whose bottom row is $\pi$ restricted to $\Q^{\dK}$. Fix a $\Q$-basis
$u_1,\dots,u_r$ of $U$ and put
\[
\vu\defeq(u_1,\dots,u_r)\in\Kf^r.
\]
For $u\in U$ the vector $u^{\Sset}$ lies in $V$, so the columns
$u_1^{\Sset},\dots,u_r^{\Sset}$ of the matrix $\vu^{\Sset}$ lie in
$V$, and its rank is at most~$d$. It therefore suffices to show that
$\rk\vu^{\Sset}\ge\nS\,r/\dK$, and we prove this for \emph{any} set
$\Tset$ of $\nT$ distinct embeddings $\Kf\to\C$ in place of $\Sset$.

Let $\cE$ be the set of all $\dK$ embeddings $\Kf\to\C$, let
$\widetilde{\Kf}\subset\C$ be the compositum of the fields
$\iota(\Kf)$, $\iota\in\cE$, which is the normal closure of $\Kf$, and
put $\galL\defeq\Gal(\widetilde{\Kf}/\Q)$. The fields and groups in
play are
\[
\begin{tikzcd}[row sep=large]
\widetilde{\Kf} \arrow[d, dash]
\arrow[ddd, dash, bend right=45, "\galL=\Gal(\widetilde{\Kf}/\Q)"',
start anchor=west, end anchor=west] \\
\Kf \arrow[d, dash, "\Sset\subset\galKk=\Gal(\Kf/k)"] \\
k \arrow[d, dash, "\dk"] \\
\Q
\end{tikzcd}
\]
with $\Kf/k$ and $\widetilde{\Kf}/\Q$ Galois, and $\Kf/\Q$ in general
not Galois. For $\iota,\iota'\in\cE$ the isomorphism
$\iota'\circ\iota^{-1}\colon\iota(\Kf)\to\iota'(\Kf)$ extends to an
element $g\in\galL$, which is to say that the diagram
\[
\begin{tikzcd}[column sep=large, row sep=large]
\widetilde{\Kf} \arrow[rr, "g"]
& & \widetilde{\Kf} \\
\iota(\Kf) \arrow[u, hook] \arrow[rr, "\iota'\circ\iota^{-1}"']
& & \iota'(\Kf) \arrow[u, hook] \\
& \Kf \arrow[lu, "\iota"] \arrow[ru, "\iota'"'] &
\end{tikzcd}
\]
commutes. Hence $\galL$ acts transitively on $\cE$ by
$\iota\mapsto g\circ\iota$, and every fiber
$\{g\in\galL:g\circ\iota=\iota'\}$ is a coset of the stabilizer of
$\iota$ and so has $[\widetilde{\Kf}:\iota(\Kf)]=|\galL|/\dK$
elements. For $\Tset'\subset\cE$ the matrix $\vu^{\Tset'}$ has entries
in $\widetilde{\Kf}$, and our end goal is to show that
$\rk\vu^{\Tset}\ge\nT\,r/\dK$. The idea is to use a double count. Some
$r$ rows of the full matrix $\vu^{\cE}$ are independent, and as $g$
runs over $\galL$ the translates $g\Tset$ sweep uniformly over the
rows, so on average $\Tset$ meets those $r$ rows in $\nT\,r/\dK$
places.

Three facts illuminate the path.

(a) $\rk\vu^{\cE}=r$. Extend $u_1,\dots,u_r$ to a $\Q$-basis
$u_1,\dots,u_\dK$ of $\Kf$. The square matrix
$(u_1,\dots,u_\dK)^{\cE}$ has nonzero determinant, its square being
the discriminant of the basis, so its columns
$u_1^{\cE},\dots,u_\dK^{\cE}$ are linearly independent; the first $r$
of them are the columns of $\vu^{\cE}$.

(b) $\rk\vu^{g\Tset'}=\rk\vu^{\Tset'}$ for $g\in\galL$. Indeed, the
diagram
\[
\begin{tikzcd}[column sep=large, row sep=large]
\widetilde{\Kf} \arrow[rr, "g"] & & \widetilde{\Kf}
& \iota(\vu) \arrow[r, "g\ \text{entrywise}"] & (g\circ\iota)(\vu) \\
& \Kf \arrow[lu, "\iota"] \arrow[ru, "g\circ\iota"'] & & &
\end{tikzcd}
\]
commutes, and therefore $\vu^{g\Tset'}$ is the entrywise image of
$\vu^{\Tset'}$ under the field automorphism $g$, which preserves rank.

(c) In view of (a), there is a set $\cB\subset\cE$ of $r$ embeddings
such that the rows $\iota(\vu)$, $\iota\in\cB$, of $\vu^{\cE}$ are
linearly independent, that is, $\rk\vu^{\cB}=r$; hence
$\rk\vu^{\Tset'}\ge|\Tset'\cap\cB|$ for every $\Tset'\subset\cE$.

\medskip Now, using (b), (c), and the fiber count, we have
\[
|\galL|\,\rk\vu^{\Tset}=\sum_{g\in\galL}\rk\vu^{g\Tset}
\ge\sum_{g\in\galL}|g\Tset\cap\cB|
=\sum_{\iota\in\Tset}\sum_{\iota'\in\cB}
\big|\{g\in\galL:g\circ\iota=\iota'\}\big|
=\nT\,r\,\frac{|\galL|}{\dK},
\]
and the proposition follows.
\end{proof}

Taking $V=\{\vzero\}$ shows that in fact $\mu(\lat)=\dK/\nS$.

\begin{corollary}\label{cor:zero}
Let $\dgb<((\bxs-1)/\nS)^{\dK/\nS}$. Then no nonzero polynomial
$\poly\in\C[z_\emb:\emb\in\Sset]$ of total degree at most $\dgb$
vanishes at every $\theta_\vn^{\Sset}$, $\vn\in\cube$.
\end{corollary}

\begin{proof}
Theorem~\ref{thm:masser} with $\kappa=\nS$, $l=\dK$, $M=\bxs-1$, and
Proposition~\ref{prop:mu}.
\end{proof}

Corollary~\ref{cor:zero} is the only output of this section used in
what follows, and the ratio $\dK/\nS$ in its exponent is the source of
the constant $|\Sset|/|\galKk|$ in Theorem~\ref{thm:main}. For
$\Kf=\Q(\sqrt d,\sqrt 2\,)$ and $\nS=3$ the exponent is the $\tfrac43$
of~\cite{Ke}; in the quadratic case of \S\ref{sec:quadratic}, where
$\Kf=F(\rou)$ and $\nS=\tfrac12[\Kf:k]+1$, it is
$2\dk[\Kf:k]/([\Kf:k]+2)$.

\section{Polynomials vanishing modulo many primes}\label{sec:poly}

In this section we construct, by the pigeonhole principle, a nonzero
polynomial that vanishes at every point $\theta_\vn^{\Sset}$,
provided that the primes with Frobenius in $\Sset$ carry enough
logarithmic mass (Lemma~\ref{lem:pigeon}). The mechanism is the
congruence $\theta^{\Nm\fp}\equiv\sigma(\theta)$ modulo the primes
above $\fp$ with Frobenius $\sigma$, which makes the value at every
point of a polynomial divisible by $(z_1^{\Nm\fp}-z_\sigma)^b$ lie in
the $b$-th power of the product of those primes.

\subsection{The Frobenius congruence}

Fix an integer $\wgt\ge 2$. A prime $\fp$ of $k$ is \emph{admissible}
if it is unramified in $\Kf$ and $\Nm\fp\ge\wgt$. For an admissible
$\fp$ and $\sigma\in\cls_\fp$ define the ideal
\be\label{eq:Q}
\fQ_{\fp,\sigma}\defeq\prod_{\fP\mid\fp,\ \Frob_{\,\fP}=\sigma}\fP
\subset\cO_\Kf .
\ee
Since distinct $\fp$ have distinct primes above them, and distinct
$\sigma$ have disjoint sets of $\fP$, the ideals \eqref{eq:Q} are
pairwise coprime.

\begin{lemma}[Frobenius]\label{lem:frob}
Let $\fp$ be admissible and $\sigma\in\cls_\fp$. Then
\be\label{eq:frob}
\theta^{\Nm\fp}\equiv\sigma(\theta)\bmod\fQ_{\fp,\sigma}
\qquad(\theta\in\cO_\Kf),
\ee
and $\Nm\fQ_{\fp,\sigma}=\Nm\fp^{\,|\galKk|/|\cls_\fp|}$.
\end{lemma}

\begin{proof}
By definition of the Frobenius element,
$\theta^{\Nm\fp}\equiv\Frob_{\,\fP}(\theta)\bmod\fP$ for every
$\theta\in\cO_\Kf$ and every $\fP\mid\fp$, and \eqref{eq:frob} follows
by the Chinese remainder theorem.

Let $f_\fp$ be the residue degree of the primes above $\fp$, which is
the order of the elements of $\cls_\fp$. Since $\fp$ is unramified,
there are $|\galKk|/f_\fp$ primes $\fP\mid\fp$, each of norm
$\Nm\fp^{f_\fp}$. The map $\fP\mapsto\Frob_{\,\fP}$ from these primes
onto $\cls_\fp$ is equivariant for the action of $\galKk$, by
transitivity on the primes and by
$\Frob_{\,g\fP}=g\Frob_{\,\fP}g^{-1}$, so all its fibers have the same
size $|\galKk|/(f_\fp|\cls_\fp|)$. Hence $\fQ_{\fp,\sigma}$ is a
product of $|\galKk|/(f_\fp|\cls_\fp|)$ primes of norm
$\Nm\fp^{f_\fp}$, and
$\Nm\fQ_{\fp,\sigma}=\Nm\fp^{\,|\galKk|/|\cls_\fp|}$.
\end{proof}

\subsection{The module of polynomials}

We work with $\nS$ complex variables $z_\emb$, $\emb\in\Sset$, and
write $z_1$ for the variable indexed by the identity of $\galKk$. Put
$\Sset^*\defeq\Sset\setminus\{1\}$. For a multi-index
$\va=(\alpha_\emb)_{\emb\in\Sset}$ of nonnegative integers we write
$\va^*\defeq(\alpha_\emb)_{\emb\in\Sset^*}$ for the multi-index with
the identity coordinate removed, and we put
\[
|\va|\defeq\sum_{\emb\in\Sset}\alpha_\emb,\qquad
\vz^{\va}\defeq\prod_{\emb\in\Sset}z_\emb^{\alpha_\emb},
\]
and likewise $|\va^*|$ and $\vz^{\va^*}$ with $\Sset^*$ in place of
$\Sset$, so that $|\va|=\alpha_1+|\va^*|$ and
$\vz^{\va}=z_1^{\alpha_1}\vz^{\va^*}$. Fix an integer $\dgb\ge 1$ and
define
\be\label{eq:idx}
\idx\defeq\{\va:\ \alpha_1+\wgt|\va^*|\le \dgb\},\qquad
\rr\defeq|\idx|,\qquad
W\defeq\sum_{\va\in\idx}|\va|.
\ee
Thus $\idx$ indexes the monomials $\vz^{\va}$ of weighted degree at
most $\dgb$, where $z_1$ has weight one and every other variable has
weight $\wgt$; the reason for the weights appears in
Lemma~\ref{lem:basis}. For a prime $\fp$ of $k$ put
\be\label{eq:E}
\Ew_\fp\defeq\sum_{\va\in\idx}\Big\lfloor\frac{\alpha_1}{\Nm\fp}\Big\rfloor
.
\ee
Let $\PolyD$ be the free $\cO_\Kf$-module with basis the monomials
$\vz^{\va}$, $\va\in\idx$; every $f\in\PolyD$ has total degree at most
$\dgb$. For $f=\sum_{\va\in\idx}c_{\va}\vz^{\va}$ in $\PolyD$ define
\[
\house{f}_\hB\defeq\max_{\va\in\idx}~\house{c_{\va}}\hB^{\,|\va|}.
\]
Since every coordinate of $\theta_\vn^{\Sset}$ lies in $\cO_\Kf$ and
has house at most $\hB$, we have
\be\label{eq:fP}
f(\theta_\vn^{\Sset})\in\cO_\Kf\mand
\house{f(\theta_\vn^{\Sset})}\le \rr\house{f}_\hB\qquad(\vn\in\cube).
\ee

\begin{lemma}\label{lem:basis}
Let $\nu\ge \wgt$ be an integer and $\sigma\in\Sset$. For $\va\in\idx$
write $\alpha_1=b\nu+v$ with $0\le v<\nu$ and put
\[
G_{\va}\defeq z_1^{v}(z_1^\nu-z_\sigma)^{b}\,\vz^{\va^*}.
\]
Then the $G_{\va}$, $\va\in\idx$, form an $\cO_\Kf$-basis of $\PolyD$.
\end{lemma}

\begin{proof}
Expanding the binomial,
\[
G_{\va}=\vz^{\va}+\sum_{j=1}^{b}(-1)^j\binom bj\,
z_1^{\alpha_1-\nu j}z_\sigma^{j}\vz^{\va^*}.
\]
For $j\ge 1$ the monomial
$z_1^{\alpha_1-\nu j}z_\sigma^{j}\vz^{\va^*}$ has weighted degree
$\alpha_1+\wgt|\va^*|-(\nu-\wgt)j\le\dgb$ if $\sigma\ne 1$, and it is
$z_1^{\alpha_1-(\nu-1)j}\vz^{\va^*}$, of weighted degree
$\alpha_1+\wgt|\va^*|-(\nu-1)j\le\dgb$, if $\sigma=1$. In either case
it is a basis monomial of $\PolyD$ of $z_1$-degree less than
$\alpha_1$. Hence, when the basis monomials are ordered by
$z_1$-degree, the matrix expressing the $G_{\va}$ in terms of them is
unitriangular.
\end{proof}

This is the purpose of the weights. The substitution
$z_1^\nu\mapsto z_\sigma$ lowers the $z_1$-degree by $\nu\ge\wgt$ and
raises $|\va^*|$ by one, so it does not leave $\PolyD$; with all
variables of weight one it would.

\begin{lemma}\label{lem:vanish}
Let $\fp$ be admissible, let $\sigma\in\Sset\cap\cls_\fp$ be fixed,
and let $G_{\va}$ be the basis of Lemma~\ref{lem:basis} for
$\nu\defeq\Nm\fp$
and this $\sigma$. Then
\[
G_{\va}(\theta_\vn^{\Sset})\in\fQ_{\fp,\sigma}^{\,\lfloor \alpha_1/\Nm\fp\rfloor}
\qquad\big(\va\in\idx,\ \vn\in\cube\big).
\]
\end{lemma}

\begin{proof}
The coordinates of $\theta_\vn^{\Sset}$ indexed by $1$ and $\sigma$
are $\theta_\vn$ and $\sigma(\theta_\vn)$, and all coordinates lie in
$\cO_\Kf$. By \eqref{eq:frob},
$\theta_\vn^{\Nm\fp}-\sigma(\theta_\vn)\in\fQ_{\fp,\sigma}$, so the
factor $(z_1^{\Nm\fp}-z_\sigma)^b$ of $G_{\va}$ contributes
$\fQ_{\fp,\sigma}^{\,b}$ with $b=\lfloor \alpha_1/\Nm\fp\rfloor$.
\end{proof}

\begin{lemma}\label{lem:residues}
Let $\cS$ be a finite set of admissible primes. For $\fp\in\cS$ put
$m_\fp\defeq\max\{1,\lfloor \dgb/\Nm\fp\rfloor\}$, and define
\be\label{eq:modulus}
\modl\defeq\prod_{\fp\in\cS}\ \prod_{\sigma\in\Sset\cap\cls_\fp}
\fQ_{\fp,\sigma}^{\,m_\fp},\qquad
\nmod\defeq\Nm\modl .
\ee
Then the vectors $(f(\theta_\vn^{\Sset})\bmod\modl)_{\vn\in\cube}$,
$f\in\PolyD$, take at most
\[
\nmod^{\,\rr}\prod_{\fp\in\cS}\ \prod_{\sigma\in\Sset\cap\cls_\fp}
\Nm\fQ_{\fp,\sigma}^{-\Ew_\fp}
\]
distinct values.
\end{lemma}

\begin{proof}
The ideals $\fQ_{\fp,\sigma}^{\,m_\fp}$ in \eqref{eq:modulus} are
pairwise coprime by \eqref{eq:Q}, hence by the Chinese remainder
theorem the vector $(f(\theta_\vn^{\Sset})\bmod\modl)_\vn$ is
determined by the vectors
$(f(\theta_\vn^{\Sset})\bmod\fQ_{\fp,\sigma}^{\,m_\fp})_\vn$;
consequently, it suffices to show that for each pair $(\fp,\sigma)$
the latter takes at most $\Nm\fQ_{\fp,\sigma}^{\,m_\fp\rr-\Ew_\fp}$
values.

For fixed $(\fp,\sigma)$, write $f=\sum_{\va\in\idx}c_{\va}G_{\va}$ in
the basis of Lemma~\ref{lem:basis} with $\nu=\Nm\fp$ and this choice
of $\sigma$. By Lemma~\ref{lem:vanish}, the residue of
$c_{\va}G_{\va}(\theta_\vn^{\Sset})$ modulo
$\fQ_{\fp,\sigma}^{\,m_\fp}$ depends only on the residue of $c_{\va}$
modulo $\fQ_{\fp,\sigma}^{\,m_\fp-\lfloor \alpha_1/\Nm\fp\rfloor}$, an
ideal that has a \emph{nonnegative} exponent since $\alpha_1\le\dgb$.
Taking into account that $|\cO_\Kf/\fQ^{a}|=\Nm\fQ^{a}$ for $a\in\N$,
the vector
$(f(\theta_\vn^{\Sset})\bmod\fQ_{\fp,\sigma}^{\,m_\fp})_\vn$ takes at
most
\[
\prod_{\va\in\idx}\Nm\fQ_{\fp,\sigma}^{\,m_\fp-\lfloor \alpha_1/\Nm\fp\rfloor}
=\Nm\fQ_{\fp,\sigma}^{\,m_\fp\rr-\Ew_\fp}
\]
values, by \eqref{eq:E}, and the result is proved.
\end{proof}

\begin{lemma}[Pigeonhole]\label{lem:pigeon}
Let $\cS$, $\modl$, and $\nmod$ be as in Lemma~\ref{lem:residues},
and recall from \eqref{eq:weight} that
$\wt_\fp=|\Sset\cap\cls_\fp|/|\cls_\fp|$. If
\be\label{eq:pigeon}
\frac1{\dk}\sum_{\fp\in\cS}\wt_\fp\,\Ew_\fp\log\Nm\fp
>W\log \hB+\rr\log(4\dK\rr\lamK),
\ee
then there is a nonzero $\poly\in\PolyD$ with
$\poly(\theta_\vn^{\Sset})=0$ for every $\vn\in\cube$.
\end{lemma}

\begin{proof}
Put $\hgt\defeq\nmod^{1/\dK}/(4\rr)$. We first count elements
$f\in\PolyD$ that have $\house{f}_\hB\le\hgt$. For each $\va\in\idx$,
the elements $c_{\va}=\sum_tu_t\omega_t$ with integers
$0\le u_t\le\hgt/(\dK\lamK\hB^{|\va|})$ are pairwise distinct, since
the $\omega_t$ are linearly independent, and satisfy
$\house{c_{\va}}\le\dK\lamK\max_tu_t\le\hgt\hB^{-|\va|}$ by
\eqref{eq:lam}. Every $f=\sum_{\va}c_{\va}\vz^{\va}$ formed in this
way has $\house{f}_\hB\le\hgt$, hence there are at least
\[
\prod_{\va\in\idx}\Big(\frac{\hgt}{\dK\lamK\hB^{|\va|}}\Big)^{\dK}
=\nmod^{\,\rr}(4\dK\rr\lamK)^{-\dK\rr}\hB^{-\dK W}
\]
of them. By Lemma~\ref{lem:residues}, these $f$ take at most
$\nmod^{\,\rr}\prod_{\fp,\sigma}\Nm\fQ_{\fp,\sigma}^{-\Ew_\fp}$
distinct residue vectors modulo $\modl$, the product over
$\fp\in\cS$ and $\sigma\in\Sset\cap\cls_\fp$. Thus, there are more
polynomials than residue vectors as soon as
\be\label{eq:count}
\prod_{\fp\in\cS}\ \prod_{\sigma\in\Sset\cap\cls_\fp}
\Nm\fQ_{\fp,\sigma}^{\,\Ew_\fp}
>(4\dK\rr\lamK)^{\dK\rr}\hB^{\dK W}.
\ee
Since $\Nm\fQ_{\fp,\sigma}=\Nm\fp^{\,|\galKk|/|\cls_\fp|}$ by
Lemma~\ref{lem:frob}, the logarithm of the left side is
\[
\sum_{\fp\in\cS}\ \sum_{\sigma\in\Sset\cap\cls_\fp}
\Ew_\fp\log\Nm\fQ_{\fp,\sigma}
=|\galKk|\sum_{\fp\in\cS}\wt_\fp\,\Ew_\fp\log\Nm\fp
=\frac{\dK}{\dk}\sum_{\fp\in\cS}\wt_\fp\,\Ew_\fp\log\Nm\fp,
\]
hence \eqref{eq:count} is equivalent to \eqref{eq:pigeon} and so
holds by hypothesis. Consequently, there are
distinct $f_1,f_2$ with $\house{f_j}_\hB\le\hgt$ and
$f_1(\theta_\vn^{\Sset})\equiv f_2(\theta_\vn^{\Sset})\bmod\modl$ for
all $\vn\in\cube$. Put $\poly\defeq f_1-f_2\ne 0$. Then
$\house{\poly}_\hB\le 2\hgt$, so \eqref{eq:fP} implies that
\[
\poly(\theta_\vn^{\Sset})\in\modl\mand
\house{\poly(\theta_\vn^{\Sset})}\le\rr\house{\poly}_\hB
\le 2\rr\hgt=\tfrac12\nmod^{1/\dK}\qquad(\vn\in\cube).
\]
But a nonzero $y\in\modl$ has
$\house{y}^{\dK}\ge|N_{\Kf/\Q}(y)|\ge\Nm\modl=\nmod$, since $\modl$
divides $y\cO_\Kf$. Hence $\poly(\theta_\vn^{\Sset})=0$ for every
$\vn\in\cube$.
\end{proof}

The factor $\dK$ in the count of polynomials and the factor
$|\galKk|=\dK/\dk$ in the norms of Lemma~\ref{lem:frob} have been
divided out of \eqref{eq:pigeon}, leaving the factor $1/\dk$ on the
left. This is the only place where the degree of the base field
enters, and it cancels in \S\ref{sec:proof} against the same factor in
$\dk\log\bxs$. From this section, \S\ref{sec:proof} uses
Lemma~\ref{lem:pigeon} and
the quantities $\rr$, $W$, and $\Ew_\fp$ of \eqref{eq:idx}
and~\eqref{eq:E}.

\section{Proof of Theorems~\ref{thm:main} and~\ref{thm:precise}}
\label{sec:proof}

Recall that $\dk=[k:\Q]$, $\dK=[\Kf:\Q]$, and $\nS=|\Sset|$. We first
record the sizes of the quantities attached to $\idx$. Put
\[
\Ic\defeq\frac{\dgb^{\nS+1}}{(\nS+1)!\,\wgt^{\nS-1}} .
\]

\begin{lemma}\label{lem:count}
Let $\nS\ge 1$ and $\wgt\ge 2$ be fixed. As $\dgb\to\infty$,
\[
\rr=\frac{(\nS+1)\Ic}{\dgb}+O(\dgb^{\nS-1}),\qquad
W=\Big(1+\frac{\nS-1}{\wgt}\Big)\Ic+O(\dgb^{\nS}),\qquad
\Ew_\fp=\frac{\Ic}{\Nm\fp}+O(\dgb^{\nS}),
\]
the last uniformly in the prime $\fp$ of $k$, with implied constants
depending only on $\nS$ and $\wgt$.
\end{lemma}

\begin{proof}
The set $\idx$ consists of the lattice points of the simplex
\[
\triangle_\wgt\defeq
\{x_1+\wgt(x_2+\dots+x_{\nS})\le\dgb\}\subset\R_{\ge 0}^{\nS},
\]
whose volume is $\dgb^{\nS}/(\nS!\,\wgt^{\nS-1})=(\nS+1)\Ic/\dgb$;
this gives the estimate for $\rr$. The substitution $y_1=x_1$,
$y_j=\wgt x_j$ for $j\ge 2$ maps $\triangle_\wgt$ onto the standard
simplex $\{y_1+\dots+y_{\nS}\le\dgb\}$ with Jacobian $\wgt^{\nS-1}$,
and $\int y_1=\dgb^{\nS+1}/(\nS+1)!$ over the standard simplex; hence
\[
\int_{\triangle_\wgt}x_1=\Ic\mand
\int_{\triangle_\wgt}x_j=\frac{\Ic}{\wgt}\qquad(2\le j\le\nS).
\]
A coordinate function is $O(\dgb)$ on $\triangle_\wgt$ and varies by
$O(1)$ on a unit cube, and the unit cubes at the points of $\idx$
cover $\triangle_\wgt$ up to a set of volume $O(\dgb^{\nS-1})$, so the
sum of a coordinate over $\idx$ differs from its integral over
$\triangle_\wgt$ by $O(\dgb^{\nS})$. This gives the estimate for
$W=\sum_{\va\in\idx}\alpha_1+\sum_{\va\in\idx}|\va^*|$. Finally,
$\Ew_\fp$ differs from $\sum_{\va\in\idx}\alpha_1/\Nm\fp$ by at most
$\rr=O(\dgb^{\nS})$.
\end{proof}

We also use the analogue of Chebyshev's bound,
\be\label{eq:cheb}
\sum_{\Nm\fp\le x}\log\Nm\fp\le 2\dk x\qquad(x\ge 1),
\ee
which holds because the primes of $k$ above a rational prime $p$
contribute at most $\dk\log p$.

\begin{proposition}\label{prop:main}
Let $\Kf/k$, $\galKk$, $\Sset$, $\dK$, $\dk$, $\nS$ be as in
\S\ref{sec:intro}, with $\nS<\dK$, and let $\wgt\ge 2$ be an integer.
Let $\bxs$ be large in terms of $\dK$ and $\wgt$, with
$\log\lamK\le\bxs$, and put
\[
\dgb\defeq\big\lceil((\bxs-1)/\nS)^{\dK/\nS}\big\rceil-1 .
\]
Then
\[
\ssum{\Nm\fp\le \dgb\\ \fp\ \text{admissible}}
\wt_\fp\,\frac{\log\Nm\fp}{\Nm\fp}
\le\Big(1+\frac{\nS-1}{\wgt}\Big)
\big(\dk\log \bxs+\dk\log\lamK\big)+O_{\dK,\wgt}(1).
\]
\end{proposition}

\begin{proof}
Every implied constant in this proof depends only on $\dK$ and
$\wgt$, as $\nS\le\dK$. Let $\cS$ be the set of admissible primes
with $\Nm\fp\le\dgb$. Every element of $\PolyD$ has degree at most
$\dgb$, so by Corollary~\ref{cor:zero} no nonzero element of $\PolyD$
vanishes at every $\theta_\vn^{\Sset}$, and by
Lemma~\ref{lem:pigeon} the inequality \eqref{eq:pigeon} fails for
$\cS$, that is,
\[
\frac1{\dk}\sum_{\fp\in\cS}\wt_\fp\,\Ew_\fp\log\Nm\fp
\le W\log\hB+\rr\log(4\dK\rr\lamK).
\]
We divide by $\Ic$ and estimate both sides by
Lemma~\ref{lem:count}. On the left, $\Ew_\fp\ge\Ic/\Nm\fp+O(\dgb^\nS)$
uniformly in $\fp$ and $0\le\wt_\fp\le 1$, so
\[
\sum_{\fp\in\cS}\wt_\fp\,\Ew_\fp\log\Nm\fp
\ge\Ic\sum_{\fp\in\cS}\wt_\fp\frac{\log\Nm\fp}{\Nm\fp}
+O\Big(\dgb^{\nS}\sum_{\Nm\fp\le\dgb}\log\Nm\fp\Big),
\]
and the error is $O(\dk\dgb^{\nS+1})=O(\dk\Ic)$ by \eqref{eq:cheb}
and $\dgb^{\nS+1}/\Ic=(\nS+1)!\,\wgt^{\nS-1}$. On the right,
\[
W/\Ic=1+(\nS-1)/\wgt+O(\dgb^{-1})
\mand
\rr\log(4\dK\rr\lamK)=O(\dgb^{\nS}\log(\dgb\lamK)),
\]
therefore
\[
\frac{W\log\hB+\rr\log(4\dK\rr\lamK)}{\Ic}
=\Big(1+\frac{\nS-1}{\wgt}\Big)\log\hB
+O\Big(\frac{\log\hB+\log(\dgb\lamK)}{\dgb}\Big),
\]
in which $\log\hB=\log\bxs+\log\lamK+\log\dK$, $\log\lamK\le\bxs$
by hypothesis, and $\dgb\gg\bxs^{\dK/\nS}$ with $\dK/\nS>1$; the
error is therefore $O(1)$. Altogether,
\[
\frac1{\dk}\sum_{\fp\in\cS}\wt_\fp\frac{\log\Nm\fp}{\Nm\fp}
\le\Big(1+\frac{\nS-1}{\wgt}\Big)
\big(\log\bxs+\log\lamK\big)+O(1),
\]
and multiplying by $\dk\le\dK$ finishes the proof.
\end{proof}

\begin{proof}[Proof of Theorem~\ref{thm:precise}]
Every implied constant in this proof depends only on $\dK$ and
$\wgt$. If $\nS=\dK$, then $k=\Q$ and $\Sset=\galKk$, so $\wt_p=1$ for
every $p$ unramified in $\Kf$, and the left side of \eqref{eq:precise}
is at most $\sum_{p\le\XX}\log p/p\le\log\XX+O(1)$ by Mertens'
theorem; this is at most the right side, since $\lamK\ge 1$. Now
assume $\nS<\dK$ and put
\[
\bxs\defeq 2+\lceil\nS(2\XX)^{\nS/\dK}\rceil .
\]
Then $((\bxs-1)/\nS)^{\dK/\nS}>2\XX$, so $\dgb\ge\XX$ in the notation
of Proposition~\ref{prop:main}, and
\[
\dk\log\bxs=\frac{\dk\nS}{\dK}\log\XX+O(1)
=\frac{\nS}{|\galKk|}\log\XX+O(1).
\]
We check the hypotheses of Proposition~\ref{prop:main}. Since
$\bxs\ge\XX^{1/\dK}$, the hypothesis on $\XX$ gives
\[
\bxs\ge\XX_0^{1/\dK}\mand\bxs\ge\log d_\Kf\ge\log\lamK,
\]
the last by \eqref{eq:minkowski}; the first bound makes $\bxs$ large
in terms of $\dK$ and $\wgt$ once $\XX_0$ is. The proposition
therefore bounds the sum over the admissible primes with
$\Nm\fp\le\XX$, as $\dgb\ge\XX$. The ramified primes contribute
nothing to the left side of \eqref{eq:precise}. The unramified primes
with $\Nm\fp<\wgt$ lie above the rational primes $p<\wgt$, at most
$\dk$ above each, and contribute at most $\dk\log p/p$ for each such
$p$, which is $O(1)$ in all. This proves \eqref{eq:precise}.
\end{proof}

\begin{proof}[Proof of Theorem~\ref{thm:main}]
Let $\eps\in(0,1]$, and define
\[
\wgt\defeq\Big\lceil\frac{4\dK}{\eps}\Big\rceil,\qquad
\expC_0\defeq\Big\lceil\frac{\dK}{\eps}\Big\rceil+1 .
\]
Let $\expC\ge\expC_0$ and $\XX\defeq d_\Kf^{\,\expC}$.
Theorem~\ref{thm:precise}
with this $\wgt$ applies as soon as
\[
\XX\ge\XX_0(\dK,\wgt)\mand\XX\ge(\log d_\Kf)^{\dK},
\]
and since $\XX=d_\Kf^{\,\expC}$, both hold once $d_\Kf$ is large in
terms of $\dK$, $\eps$, and $\expC$; assume that this is the case. By
\eqref{eq:minkowski}, we have
\[
\dk\log\lamK\le\frac{\dK}{2}\log d_\Kf=\frac{\dK}{2\expC}\log\XX,
\]
and $\nS-1\le\dK$, hence Theorem~\ref{thm:precise} gives
\[
\sum_{\Nm\fp\le\XX}\wt_\fp\,\frac{\log\Nm\fp}{\Nm\fp}
\le\Big(1+\frac{\dK}{\wgt}\Big)
\Big(\frac{\nS}{|\galKk|}+\frac{\dK}{2\expC}\Big)\log\XX
+O_{\dK,\eps,\expC}(1).
\]
It remains to bound the coefficient of $\log\XX$. We have
$\dK/\wgt\le\eps/4$ and $\dK/(2\expC)<\eps/2$ by the choice of
$\wgt$ and $\expC_0$, and $\nS\le|\galKk|$, so
\[
\Big(1+\frac{\dK}{\wgt}\Big)\Big(\frac{\nS}{|\galKk|}+\frac{\dK}{2\expC}\Big)
\le\Big(1+\frac{\eps}{4}\Big)\Big(\frac{\nS}{|\galKk|}+\frac{\eps}{2}\Big)
\le\frac{\nS}{|\galKk|}+\frac{\eps}{2}+\frac{\eps}{4}\cdot\frac32
<\frac{\nS}{|\galKk|}+\eps .
\]
The remaining fields $\Kf$ have bounded discriminant, so there are
finitely many of them by Hermite's theorem, and for each of them the
bound holds with the constant $\nS/|\galKk|$ and an $O_\Kf(1)$ by
\eqref{eq:chebotarev}. This completes the proof.
\end{proof}

\section{Remarks}\label{sec:remarks}

\begin{remark}\label{rem:compare}
The bound of Theorem~\ref{thm:main} is one-sided, and the identity
decides the side. The hypothesis $1\in\Sset$ is a normalization,
since the argument with $\emb_0(\theta)$ as the distinguished
variable, $\emb_0\in\Sset$, bounds the mass of $\Sset\emb_0^{-1}$,
which contains the identity. By Mertens' theorem for $k$,
Theorem~\ref{thm:main} gives the lower bound
$\big(|\Sset'|/|\galKk|-\eps\big)\log\XX$ for every
$\Sset'\subset\galKk$ with $1\notin\Sset'$, but no upper bound for
such an $\Sset'$. This is a limitation of the method and not of the
problem. The determinant method, as used for example in~\cite{Gr,Ko},
imposes vanishing to high order in one variable at the pairs
$(\alpha,\sigma\alpha)$, with $\alpha$ running over about $2\XX$ small
integers of $\Kf$; the derivatives supply the multiplicity that our
box of $\bxs^{\dK}$ points supplies, the second variable carries
degree only $\XX/\log\XX$, and the identity plays no part. The bound
of~\cite{Ko} therefore holds for every conjugacy class, with the exact
density, for every $\XX\ge 2$, and with absolute constants, where
Theorem~\ref{thm:main} needs $\XX\ge d_\Kf^{\,\expC_0}$ and gives the
density only up to $\eps$. The theorem of~\cite{Ko} is stated for
$\Kf/\Q$ Galois, but its proof uses only that $\sigma$ is an
automorphism of $\Kf$ and that the $\alpha$ are integers of $\Kf$. Run
over a base field $k$, with $\Nm\fp$ in place of $p$ in the congruence
and in the row operation, the norm of the determinant still taken to
$\Q$, and Chebyshev's bound for $k$ in place of
$\vartheta(\XX)\ll\XX$, it gives for $\Kf/k$ Galois and a class $\cls$
of size $c$ the bound $\frac{c}{|\galKk|}(\log\XX+\tfrac12\log d_\Kf)$
up to $O(c\,\dk+\dk\log\log\XX)$, which contains
Theorem~\ref{thm:quadratic} with $\Kf=F$ in place of $F(\rou)$. In our
method the identity is the cost. With $\Kf=F$ and $\Sset=\galKk$ the
constant is one, and the root of unity dilutes it to
$\tfrac12+1/[\Kf:k]$.
\end{remark}

\begin{remark}
Nothing is used about $\Q(\rou)$ beyond its being a Galois extension
of $\Q$ of degree $\auxp-1$ in which $\rou$ is an integer of house
one. For any field $E$ Galois over $k$ with $E\cap F=k$ and
$[E:k]=g$, and $\Kf=FE$ with $\Sset$ the identity together with the
coset of $\Gal(\Kf/F)$, the same proof gives the constant
$(g+1)/(2g)$; the case $g=2$ is the $\sqrt 2$ of~\cite{Ke} and gives
$\tfrac34$.
\end{remark}

\begin{remark}
The weights \eqref{eq:weight} are forced by the mechanism. The
congruence $\theta^{\Nm\fp}\equiv\sigma(\theta)$ holds modulo the
primes above $\fp$ whose Frobenius is $\sigma$ itself, of combined
norm $\Nm\fp^{\,|\galKk|/|\cls_\fp|}$, so a class of size $c$ meeting
$\Sset$ in one element contributes with weight $1/c$. For
$\Sset=\{1,\sigma\}$ with $\sigma$ in a class $\cls$ of size $c$,
Theorem~\ref{thm:main} reads
\[
\sum_{\cls_\fp=\{1\}}\frac{\log\Nm\fp}{\Nm\fp}
+\frac1c\sum_{\cls_\fp=\cls}\frac{\log\Nm\fp}{\Nm\fp}
\le\Big(\frac{2}{|\galKk|}+\eps\Big)\log\XX+O(1),
\]
sharp by \eqref{eq:chebotarev}, whereas the unweighted sum over
$\{1\}\cup\cls$ has density $(1+c)/|\galKk|$.
\end{remark}

\begin{remark}
Everything is effective. A large $|\Sset|$ costs a large $d_\Kf$,
since the error terms in Proposition~\ref{prop:main} carry constants
of size $\nS!\,\wgt^{\nS}$; an explicit $c(k)$ in
Corollary~\ref{cor:siegel} would fix a small~$\auxp$ and proceed as
in~\cite{Ba}. An explicit constant in Theorem~\ref{thm:bs} needs
in addition the constant of Lemma~\ref{lem:residue}, which Stark's
proof makes explicit, and numerical values of $\kappa_k$ and of the
real zeros of $\zeta_k(s)$ in $(0,1)$ for the given $k$.
\end{remark}

\section*{Acknowledgments}

The author thanks the Max Planck Institute for Mathematics in
Bonn for its support and excellent working conditions during the
preparation of this paper, Bryce Kerr for many helpful
discussions, Ben Green for comments on an earlier version, and
Christopher Frei, Wataru Kai, and Jesse Thorner for correspondence
about their work.

The large language model Claude (Anthropic) was used throughout the
preparation of this paper. The key ideas came from the author, who
prompted Claude with the question considered here and with the
papers~\cite{Ba} and~\cite{Ke}; from these inputs Claude produced an
initial argument. The author then worked through every line of the
mathematics by hand, verifying the arguments in their entirety, and
wrote the final version of the paper, supplying its notation,
elaborations of the underlying arguments, and the diagrams. Claude
was also used for polishing the wording, for literature searches, and
for proof checking. The author takes full responsibility for the
content.


\begin{thebibliography}{99}

\bibitem{Ba}
W.~D. Banks, On the OpenAI proof that Landau--Siegel zeros do not
exist. arXiv:2610.15483.

\bibitem{Ca}
J.~W.~S. Cassels, \emph{An introduction to the geometry of numbers}.
Grundlehren der mathematischen Wissenschaften, vol.~99, Springer,
Berlin, 1959.

\bibitem{CT}
D.~Cureton and J.~Thorner, A quasi-Riemann hypothesis for Dedekind
zeta functions. Hexagon preprint hexagon:2610.00200v2, 2026.

\bibitem{Da}
H.~Davenport, \emph{Multiplicative number theory}. 2nd ed., revised by
H.~L. Montgomery, Graduate Texts in Mathematics, vol.~74, Springer,
New York, 1980.

\bibitem{FK}
C.~Frei and W.~Kai, Exceptional zeros of Hecke $L$-functions. Hexagon
preprint hexagon:2610.00466v1, 2026.

\bibitem{GL}
S.~R. Garcia and E.~S. Lee, Unconditional explicit Mertens' theorems
for number fields and Dedekind zeta residue bounds.
\emph{Ramanujan J.} 57 (2022), 1169--1191.

\bibitem{Gr}
B.~Green, On Landau--Siegel zeros. Hexagon preprint
hexagon:2610.00311v2, 2026.

\bibitem{Ke}
B.~Kerr, A simpler proof that there are no Landau--Siegel zeros.
arXiv:2610.21099.

\bibitem{Ko}
S.~D. Kominers, Small algebraic integers and Frobenius prime sums.
Preprint, October 2026.

\bibitem{LO}
J.~C. Lagarias and A.~M. Odlyzko, Effective versions of the Chebotarev
density theorem. In \emph{Algebraic number fields} (Durham, 1975),
A.~Fr\"ohlich, ed., Academic Press, 1977, 409--464.

\bibitem{Li}
B.~Liu, Zero-free half-planes for Hecke $L$-functions over number
fields. arXiv:2610.26835.

\bibitem{Lo}
S.~Louboutin, Explicit bounds for residues of Dedekind zeta
functions, values of $L$-functions at $s=1$, and relative class
numbers. \emph{J. Number Theory} 85 (2000), 263--282.

\bibitem{Ma}
D.~W. Masser, On polynomials and exponential polynomials in several
complex variables. \emph{Invent. Math.} 63 (1981), 81--95.

\bibitem{Ne}
J.~Neukirch, \emph{Algebraic number theory}. Grundlehren der
mathematischen Wissenschaften, vol.~322, Springer, Berlin, 1999.

\bibitem{OAI}
OpenAI, Uniform exclusion of Landau--Siegel zeros. Preprint, 1 October
2026.

\bibitem{Ra}
H.~Rademacher, On the Phragm\'en--Lindel\"of theorem and some
applications. \emph{Math. Z.} 72 (1959), 192--204.

\bibitem{Ro}
M.~Rosen, A generalization of Mertens' theorem. \emph{J. Ramanujan
Math. Soc.} 14 (1999), 1--19.

\bibitem{St}
H.~M. Stark, Some effective cases of the Brauer--Siegel theorem.
\emph{Invent. Math.} 23 (1974), 135--152.

\end{thebibliography}
\end{document}